\def\DateTime{26/June/2012, 9:00 (Kyoto)}
\def\Version{Version $2.1$}
\def\yes{\if00}
\def\no{\if01}
\def\iftenpt{\no}
\def\ifelevenpt{\no}
\def\iftwelvept{\yes}

\def\ifusepdf{\no}
\def\ifpsfont{\no}
\def\ifpxfont{\yes}
\def\ifquery{\no}

\iftenpt
\documentclass[leqno]{amsart}
\else\fi
\ifelevenpt
\documentclass[leqno,11pt]{amsart}
\else\fi
\iftwelvept
\documentclass[leqno,12pt]{amsart}
\else\fi

\usepackage{amssymb}
\usepackage{amscd}
\usepackage{latexsym}
\usepackage{verbatim}
\usepackage{mathrsfs}
\usepackage[all]{xy}

\ifusepdf
\usepackage{hyperref}
\else\fi
\ifpsfont
\usepackage[T1]{fontenc}
\usepackage{times}
\else\fi
\ifpxfont
\usepackage{pxfonts}
\else\fi

\ifelevenpt
\else\fi

\iftwelvept
\else\fi

\theoremstyle{plain}
\newtheorem{Theorem}{Theorem}[section]
\newtheorem{Proposition}[Theorem]{Proposition}
\newtheorem{Lemma}[Theorem]{Lemma}

\newtheorem{Corollary}[Theorem]{Corollary}

\newtheorem{Claim}{Claim}[Theorem]

\theoremstyle{definition}

\newtheorem{Remark}[Theorem]{Remark}
\newtheorem{Example}[Theorem]{Example}

\def\rom{\textup}
\newcommand{\ZZ}{{\mathbb{Z}}}
\newcommand{\QQ}{{\mathbb{Q}}}
\newcommand{\RR}{{\mathbb{R}}}
\newcommand{\KK}{{\mathbb{K}}}
\newcommand{\CC}{{\mathbb{C}}}
\newcommand{\PP}{{\mathbb{P}}}

\newcommand{\OO}{{\mathscr{O}}}

\newcommand{\Proj}{\operatorname{Proj}}

\newcommand{\Rat}{\operatorname{Rat}}

\newcommand{\Spec}{\operatorname{Spec}}

\newcommand{\Supp}{\operatorname{Supp}}

\newcommand{\mult}{\operatorname{mult}}

\newcommand{\adeg}{\widehat{\operatorname{deg}}}

\newcommand{\ord}{\operatorname{ord}}
\newcommand{\Div}{\operatorname{Div}}
\newcommand{\aDiv}{\widehat{\operatorname{Div}}}

\newcommand{\avol}{\widehat{\operatorname{vol}}}
\newcommand{\acvol}{\widehat{\operatorname{vol}}_{\chi}}
\newcommand{\aH}{\hat{H}^0}

\newcommand{\vol}{\operatorname{vol}}
\newcommand{\aew}{({\rm a.e.})}

\newcommand{\rest}[2]{\left.{#1}\right\vert_{{#2}}}
\ifquery
\def\query#1{\setlength\marginparwidth{65pt} 
\marginpar{\raggedright\fontsize{7.81}{9} 
\selectfont\upshape\hrule\smallskip 
#1\par\smallskip\hrule}} 

\else
\def\query#1{}
\fi

\begin{document}

\title[Numerical characterization of nef arithmetic divisors]%
{Numerical characterization of nef arithmetic divisors \\
on arithmetic surfaces}
\author{Atsushi Moriwaki}
\address{Department of Mathematics, Faculty of Science,
Kyoto University, Kyoto, 606-8502, Japan}
\email{moriwaki@math.kyoto-u.ac.jp}
\date{\DateTime, (\Version)}
\begin{abstract}
In this paper, we give a numerical characterization of 
nef arithmetic $\RR$-Cartier divisors of $C^0$-type on an arithmetic surface.
Namely an arithmetic $\RR$-Cartier divisor $\overline{D}$ of $C^0$-type is nef if and only if
$\overline{D}$ is pseudo-effective and $\adeg(\overline{D}^2) = \avol(\overline{D})$.
\end{abstract}


\maketitle



\section*{Introduction}

Let $X$ be a generically smooth, normal and projective arithmetic surface and
let $X \to \Spec(O_K)$ be the Stein factorization of $X \to \Spec(\ZZ)$,
where $K$ is a number field and $O_K$ is the ring of integers in $K$.
Let $\overline{L}$ be an arithmetic divisor of $C^{\infty}$-type on $X$
with $\deg(L_K) = 0$ (cf. Conventions and terminology~\ref{CT:arith:div}).
Faltings-Hriljac's Hodge index theorem (\cite{FaCAS}, \cite{Hri}) says that
\[
\adeg(\overline{L}^2) \leq 0
\]
and the equality holds if and only if
$\overline{L} = \widehat{(\phi)} + (0, \eta)$ for
some $F_{\infty}$-invariant locally constant real valued function $\eta$ on $X(\CC)$
and $\phi \in \Rat(X)^{\times}_{\QQ} := \Rat(X)^{\times} \otimes_{\ZZ} \QQ$.
The inequality part of their Hodge index theorem can be generalized as follows:
Let $\overline{D}$ be an integrable arithmetic $\RR$-Cartier divisor of
$C^0$-type on $X$, that is, $\overline{D} = \overline{P} - \overline{Q}$ for some
nef arithmetic $\RR$-Cartier divisors $\overline{P}$ and $\overline{Q}$ of $C^0$-type.
(cf. Conventions and terminology~\ref{CT:arith:div} and \ref{CT:Div:Nef}).
If $\deg(D_K) \geq 0$, 
then
\[
\adeg (\overline{D}^2) \leq \avol(\overline{D})
\]
(cf. \cite[Theorem~6.2]{MoCont}, \cite[Theorem~6.6.1]{MoArZariski}, Theorem~\ref{thm:equal:cond:gen:Hodge}). 
This inequality is called the {\em generalized Hodge index theorem}.
It is very interesting to ask the equality condition of the inequality.
It is known that if $\overline{D}$ is nef, then 
$\adeg (\overline{D}^2) = \avol(\overline{D})$
(cf.  \cite[Corollary~5.5]{MoCont},
\cite[Proposition-Definition~6.4.1]{MoArZariski}), so that the problem is the converse.
In the case where $\deg(D_K) = 0$ (and hence $\avol(\overline{D})=0$), it is nothing more than the equality condition of
the Hodge index theorem (cf. Lemma~\ref{lem:nef:deg:zero}).
Thus the following theorem gives an answer to
the above question.

\begin{Theorem}[cf. Theorem~\ref{thm:equal:cond:gen:Hodge}]
\label{thm:GHIT:equal:cond}
We assume that 
$\deg(D_K) > 0$.
Then $\overline{D}$ is nef if and only if
$\adeg(\overline{D}^2) = \avol(\overline{D})$.
\end{Theorem}

For the proof of the above theorem, we need the integral formulae of the arithmetic volumes due to
Boucksom-Chen \cite{BC} and the existence of the Zariski decomposition of big arithmetic divisors \cite{MoArZariski}.
From the point of view of a characterization of nef arithmetic $\RR$-Cartier divisors,
the following variant of the above theorem is also significant.

\begin{Corollary}[cf. Corollary~\ref{cor:characterization:nef:ar:div}]
$\overline{D}$ is nef if and only if $\overline{D}$ is pseudo-effective
and $\adeg(\overline{D}^2) = \avol(\overline{D})$.
\end{Corollary}

Let $\Upsilon(\overline{D})$ be the set of all
arithmetic $\RR$-Cartier divisors $\overline{M}$ of $C^0$-type on $X$ such that $\overline{M}$ is nef and
$\overline{M} \leq \overline{D}$.
As an application of the above theorem, we have the following numerical
characterization of the greatest element of $\Upsilon(\overline{D})$.

\begin{Corollary}[cf. Corollary~\ref{cor:characterization:Zariski:decomp}]
We assume that $X$ is regular. 
Let $\overline{P}$ be an arithmetic $\RR$-Cartier divisor of $C^0$-type on $X$.
Then the following are equivalent:
\begin{enumerate}
\renewcommand{\labelenumi}{(\arabic{enumi})}
\item
$\overline{P}$ is the greatest element of $\Upsilon(\overline{D})$, that is,
$\overline{P} \in \Upsilon(\overline{D})$ and $\overline{M} \leq \overline{P}$ 
for all $\overline{M} \in \Upsilon(\overline{D})$.

\item
$\overline{P}$ is an element of $\Upsilon(\overline{D})$ with the following property:
\[
\adeg(\overline{P} \cdot \overline{B}) = 0\quad\text{and}\quad
\adeg(\overline{B}^2) < 0
\]
for all integrable arithmetic $\RR$-Cartier divisors $\overline{B}$ of
$C^0$-type with $(0,0) \lneqq \overline{B} \leq \overline{D} - \overline{P}$
\rom{(}cf. Conventions and terminology~\rom{\ref{CT:Div:Nef}}\rom{)}.
\end{enumerate}
\end{Corollary}

Finally I would like to thank Prof. Yuan and Prof. Zhang for their helpful comments.

\renewcommand{\thesubsubsection}{\arabic{subsubsection}}
\renewcommand{\theequation}{CT.\arabic{subsubsection}.\arabic{Claim}}
\subsection*{Conventions and terminology}
Here we fix several conventions and the terminology of this paper.
An {\em arithmetic variety} means a quasi-projective and flat integral scheme over $\ZZ$.
It is said to be
{\em generically smooth} if the generic fiber over $\ZZ$ is smooth over $\QQ$.
Throughout this paper,
$X$ is a $(d+1)$-dimensional,
generically smooth, normal and projective arithmetic variety.
Let $X \to \Spec(O_K)$ be the Stein factorization of $X \to \Spec(\ZZ)$,
where $K$ is a number field and $O_K$ is the ring of integers in $K$.
For details of the following \ref{CT:arith:div} and \ref{CT:positivity:arithmetic:divisors},
see \cite{MoArZariski} and \cite{MoD}.

\subsubsection{}
\label{CT:normed:module}
A pair $(M, \Vert\cdot\Vert)$ is called a {\em normed $\ZZ$-module} if
$M$ is a finitely generated $\ZZ$-module and
$\Vert\cdot\Vert$ is a norm of $M_{\RR} := M \otimes_{\ZZ} \RR$.
A quantity
\[
\log \left(\frac{\vol\left(\{ x \in M_{\RR} \mid \Vert x \Vert \leq 1 \}\right)}
{\vol(M_{\RR}/(M/M_{tor}))} \right) + \log \#(M_{tor})
\]
does not depend on the choice of the Haar measure $\vol$ 
on $M_{\RR}$, where $M_{tor}$ is the group of torsion elements of $M$.
We denote the above quantity by $\hat{\chi}(M, \Vert\cdot\Vert)$.

\subsubsection{}
\label{CT:arith:div}
Let $\KK$ be either $\QQ$ or $\RR$.
Let $\Div(X)$ be the group of Cartier divisors on $X$ and let
$\Div(X)_{\KK} := \Div(X) \otimes_{\ZZ} \KK$, whose element is
called a {\em $\KK$-Cartier divisor on $X$}.
For $D \in \Div(X)_{\RR}$,
we define $H^0(X, D)$ and $H^0(X_K, D_K)$ to be
\[
\begin{cases}
H^0(X, D) = \left\{ \phi \in \Rat(X)^{\times} \mid D + (\phi) \geq 0 \right\} \cup \{ 0 \},\\
H^0(X_K, D_K) = \left\{ \phi \in \Rat(X_K)^{\times} \mid \text{$D_K + (\phi)_K \geq 0$ on $X_K$} \right\} \cup \{ 0 \},
\end{cases}
\]
where $X_K$ is the generic fiber of $X \to \Spec(O_K)$.

A pair $\overline{D} = (D, g)$ 
is called an {\em arithmetic $\KK$-Cartier divisor
of $C^{\infty}$-type} (resp. {\em of $C^{0}$-type}) if
the following conditions are satisfied:
\begin{enumerate}
\renewcommand{\labelenumi}{(\alph{enumi})}
\item
$D$ is a $\KK$-Cartier divisor on $X$, that is,
$D = \sum_{i=1}^r a_i D_i$  for some $D_1, \ldots, D_r \in \Div(X)$ and
$a_1, \ldots, a_r \in \KK$.

\item
$g : X(\CC) \to \RR \cup \{\pm\infty\}$ is a locally integrable function and
$g \circ F_{\infty} = g \ (a.e.)$, where $F_{\infty} : X(\CC) \to X(\CC)$
is the complex conjugation map.

\item
For any point $x \in X(\CC)$, there exist an open neighborhood $U_{x}$ of $x$ and
a $C^{\infty}$-function (resp. continuous function) 
$u_x$ on $U_x$ such that
\[
g = u_x + \sum_{i=1}^r (-a_i) \log \vert f_i \vert^2\quad(a.e.)
\]
on $U_x$,
where $f_i$ is a local equation of $D_i$ over $U_x$ for each $i$.
\end{enumerate}
The function $g$ is called a {\em $D$-Green function of $C^{\infty}$-type}
(resp. {\em of $C^0$-type}).
Note that $dd^c([u_x])$ does not depend on the choice of local equations $f_1, \ldots, f_r$,
so that $dd^c([u_x])$ is defined globally on $X(\CC)$. It is called the {\em first Chern current of $\overline{D}$}
and is denoted by $c_1(\overline{D})$, that is,
$c_1(\overline{D}) = dd^c([g]) + \delta_D$.
Note that, if $\overline{D}$ is of $C^{\infty}$-type, then
$c_1(\overline{D})$ is represented by a $C^{\infty}$-form,
which is called the {\em first Chern form of $\overline{D}$}.
Let $\mathcal{C}$ be either $C^{\infty}$ or $C^0$.
The set of all arithmetic $\KK$-Cartier divisors of $\mathcal{C}$-type is denoted by
$\aDiv_{\mathcal{C}}(X)_{\KK}$. Moreover, the group
\[
\left\{ (D, g) \in \aDiv_{\mathcal{C}}(X)_{\QQ}
\mid D \in \Div(X) \right\}
\]
is denoted by
$\aDiv_{\mathcal{C}}(X)$. An element of $\aDiv_{\mathcal{C}}(X)$ is called
an {\em arithmetic Cartier divisor
of $\mathcal{C}$-type}.
For $\overline{D} = (D, g), \overline{E} = (E, h) \in \aDiv_{C^0}(X)_{\KK}$,
we define relations
$\overline{D} = \overline{E}$ and $\overline{D} \geq \overline{E}$ as follows:
\begin{align*}
\overline{D} = \overline{E} & \quad\overset{\text{def}}{\Longleftrightarrow}\quad D = E, \ \  g = h \ (a.e.), \\
\overline{D} \geq \overline{E} & \quad\overset{\text{def}}{\Longleftrightarrow}\quad D \geq E, \ \  g \geq h \ (a.e.).
\end{align*}
Let $\Rat(X)^{\times}_{\KK} := \Rat(X)^{\times} \otimes_{\ZZ} \KK$, and
let 
\[
(\ )_{\KK} : \Rat(X)^{\times}_{\KK} \to \Div(X)_{\KK}\quad\text{and}\quad
\widehat{(\ )}_{\KK} : \Rat(X)^{\times}_{\KK} \to \aDiv_{C^{\infty}}(X)_{\KK}
\]
be the natural extensions of the homomorphisms 
\[
\Rat(X)^{\times} \to \Div(X)\quad\text{and}\quad
\Rat(X)^{\times} \to \aDiv_{C^{\infty}}(X)
\]
given by
$\phi \mapsto (\phi)$ and 
$\phi \mapsto \widehat{(\phi)}$ respectively.
Let $\overline{D}$ be an arithmetic $\RR$-Cartier divisor of $C^0$-type.
We define $\widehat{\Gamma}^{\times}(X, \overline{D})$ and $\widehat{\Gamma}^{\times}_{\KK}(X, \overline{D})$
to be
\[
\begin{cases}
\widehat{\Gamma}^{\times}(X, \overline{D}) := \left\{ \phi \in \Rat(X)^{\times} \mid
\overline{D} + \widehat{(\phi)} \geq (0,0) \right\}, \\
\widehat{\Gamma}^{\times}_{\KK}(X, \overline{D}) := \left\{ \phi \in \Rat(X)^{\times}_{\KK} \mid
\overline{D} + \widehat{(\phi)}_{\KK} \geq (0,0) \right\}.
\end{cases}
\]
Note that $\widehat{\Gamma}^{\times}_{\QQ}(X, \overline{D}) = \bigcup_{n=1}^{\infty} \widehat{\Gamma}^{\times}(X, n\overline{D})^{1/n}$.
Moreover, we set
\[
\aH(X, \overline{D}) := \widehat{\Gamma}^{\times}(X, \overline{D}) \cup \{ 0 \}
\quad\text{and}\quad
\aH_{\KK}(X, \overline{D}) := \widehat{\Gamma}^{\times}_{\KK}(X, \overline{D})
\cup \{ 0 \}.
\]
For $\xi \in X$,
we define the {\em $\KK$-asymptotic multiplicity of $\overline{D}$ at $\xi$} to be
\[
\mu_{\KK,\xi}(\overline{D}) :=
\begin{cases}
\inf \left\{ \mult_{\xi}(D + (\phi)_{\KK}) \mid \phi \in \widehat{\Gamma}_{\KK}^{\times}(X, \overline{D}) \right\} &
\text{if $\widehat{\Gamma}_{\KK}^{\times}(X, \overline{D}) \not= \emptyset$}, \\
\infty & \text{otherwise},
\end{cases}
\]
(for details, see \cite[Proposition~6.5.2, Proposition~6.5.3]{MoArZariski} and
\cite[Section~2]{MoD}).

\subsubsection{}
\label{CT:volume}
Let $\overline{D} = (D, g)$ be an arithmetic $\RR$-Cartier divisor 
of $C^0$-type on $X$.
Let $\phi \in H^0(X(\CC), D_{\CC})$, that is,
$\phi \in \Rat(X(\CC))^{\times}$ and $(\phi) + D_{\CC} \geq 0$ on $X(\CC)$.
Then
$\vert \phi \vert \exp(-g/2)$ is represented by a continuous function $\vert \phi \vert_g^c$
on $X(\CC)$ (cf. \cite[SubSection~2.5]{MoArZariski}), so that we may consider
$\sup \{ \vert \phi \vert_g^c(x) \mid x \in X(\CC) \}$. We denote it
by $\Vert \phi \Vert_{\overline{D}}$ or $\Vert \phi \Vert_{g}$.
Note that, for $\phi \in H^0(X, D)$,
$\phi \in \aH(X, \overline{D})$ if and only if $\Vert \phi \Vert_{\overline{D}} \leq 1$.
We define $\avol(\overline{D})$ and $\acvol(\overline{D})$ to be
\[
\avol(\overline{D}) := \limsup_{m \to\infty}
\frac{ \log \# \aH(X, m\overline{D}) }{m^{d+1}/(d+1)!},\quad
\acvol(\overline{D}) := \limsup_{m\to\infty}
\frac{ \hat{\chi}(H^0(X, mD), \Vert\cdot\Vert_{m\overline{D}})}{m^{d+1}/(d+1)!}.
\]
It is well known that $\avol(\overline{D}) \geq \acvol(\overline{D})$.
More generally,
for $\xi_1, \ldots, \xi_l \in X$ and $\mu_1, \ldots, \mu_l \in \RR_{\geq 0}$,
we define $\avol(\overline{D}; \mu_1 \xi_1, \ldots, \mu_l \xi_l)$ to be
\begin{multline*}
\avol(\overline{D}; \mu_1 \xi_1, \ldots, \mu_l \xi_l) :=  \\
\limsup_{m \to\infty}
\frac{\log \# \left(  \left\{ \phi \in \widehat{\Gamma}^{\times}(X, m\overline{D})
 \mid \mult_{\xi_i}(mD + (\phi)) \geq \mu_i\ (\forall i) \right\}\cup \{ 0 \}\right)}{m^{d+1}/(d+1)!}.
\end{multline*}
Note that
$\avol(\overline{D}; \mu\xi) = \avol(\overline{D})$ for $0 \leq \mu \leq \mu_{\QQ, \xi}(\overline{D})$.

\subsubsection{}
\label{CT:positivity:arithmetic:divisors}
Let $\overline{D}$ be an arithmetic $\RR$-Cartier divisor of $C^0$-type on $X$.
The effectivity, bigness, pseudo-effectivity and nefness of $\overline{D}$ are defined as follows:
\begin{enumerate}
\item[$\bullet$]
$\overline{D}$ is effective
$\quad\overset{\text{def}}{\Longleftrightarrow}\quad$
$\overline{D} \geq (0, 0)$.

\item[$\bullet$] $\overline{D}$ is big $\quad\overset{\text{def}}{\Longleftrightarrow}\quad$ $\avol(\overline{D}) > 0$.

\item[$\bullet$] $\overline{D}$ is pseudo-effective $\quad\overset{\text{def}}{\Longleftrightarrow}\quad$ $\overline{D}+\overline{A}$
is big for any big arithmetic $\RR$-Cartier divisor $\overline{A}$ of $C^0$-type.

\item[$\bullet$] $\overline{D} = (D,g)$ is nef $\quad\overset{\text{def}}{\Longleftrightarrow}$ 
\par
\begin{enumerate}
\renewcommand{\labelenumii}{(\alph{enumii})}
\item
$\adeg(\rest{\overline{D}}{C}) \geq 0$ for all reduced and irreducible $1$-dimensional closed subschemes $C$ of $X$.

\item
$c_1(\overline{D})$ is a positive current.
\end{enumerate}
\end{enumerate}
A decomposition 
$\overline{D} = \overline{P} + \overline{N}$ is called a {\em Zariski decomposition of $\overline{D}$}
if the following properties are satisfied:
\begin{enumerate}
\renewcommand{\labelenumi}{(\arabic{enumi})}
\item
$\overline{P}$ and $\overline{N}$ are arithmetic $\RR$-Cartier divisors of $C^0$-type on $X$.

\item
$\overline{P}$ is nef and $\overline{N}$ is effective.

\item
$\avol(\overline{P}) = \avol(\overline{D})$.
\end{enumerate}
We set
\[
\Upsilon(\overline{D}) := 
\left\{ \overline{M} \ \left| \ \begin{array}{l}
\text{$\overline{M}$ is an
arithmetic $\RR$-Cartier divisor of $C^0$-type} \\
\text{such that $\overline{M}$ is nef and $\overline{M} \leq \overline{D}$}
\end{array}
\right\}\right..
\]
If $\overline{P}$ is the greatest element of $\Upsilon(\overline{D})$ (i.e.
$\overline{P} \in \Upsilon(\overline{D})$ and $\overline{M} \leq \overline{P}$ 
for all $\overline{M} \in \Upsilon(\overline{D})$) and $\overline{N} = \overline{D} - \overline{P}$,
then $\overline{D} = \overline{P} + \overline{N}$ is a Zariski decomposition of $\overline{D}$
(cf. Proposition~\ref{prop:small:sec:D:P}).

\subsubsection{}
\label{CT:Div:Nef}
Let $\overline{D}$ be an arithmetic $\RR$-Cartier divisor of $C^0$-type on $X$.
According to \cite{ZhSmall},
we say $\overline{D}$ is {\em integrable} if there are
nef arithmetic $\RR$-Cartier divisors $\overline{P}$ and $\overline{Q}$ of $C^0$-type
such that $\overline{D} = \overline{P} - \overline{Q}$.
Note that if either $\overline{D}$ is of $C^{\infty}$-type, or $c_1(\overline{D})$ is a positive current,
then $\overline{D}$ is integrable (cf. \cite[Proposition~6.4.2]{MoArZariski}).
Moreover, for integrable arithmetic $\RR$-Cartier divisors $\overline{D}_0, \ldots, \overline{D}_d$ of $C^0$-type on $X$,
the arithmetic intersection number
$\adeg(\overline{D}_0 \cdots \overline{D}_d)$ is defined in the natural way
(cf. \cite[SubSection~6.4]{MoArZariski}, \cite[SubSection~2.1]{MoD}).
Note that if $\overline{D} = \overline{P} + \overline{N}$ is a Zariski decomposition and
$\overline{D}$ is integrable,
then $\overline{N}$ is also integrable.

\subsubsection{}
\label{CT:max:arith:div}
We assume that $X$ is regular and $d=1$.
Let $D_1, \ldots, D_k$ be $\RR$-Cartier divisors on $X$.
We set $D_i = \sum_{C} a_{i,C} C$ for each $i$,
where $C$ runs over all reduced and irreducible $1$-dimensional closed subschemes on $X$. We define $\max \{ D_1, \ldots, D_k \}$ to be
\[
\max \{ D_1, \ldots, D_k  \} := \sum_C \max \{ a_{1, C}, \ldots, a_{k, C} \} C.
\]
Let $\overline{D}_1 =(D_1, g_1), \ldots, \overline{D}_k = (D_k, g_k)$ be arithmetic $\RR$-Cartier divisors of $C^0$-type
on $X$. Then $\max \{ \overline{D}_1, \ldots, \overline{D}_k \}$ is defined to be
\[
\max \{ \overline{D}_1, \ldots, \overline{D}_k \} := \left(  \max \{ D_1, \ldots, D_k \}, \max \{ g_1, \ldots, g_k \} \right).
\]
Note that $\max \{ \overline{D}_1, \ldots, \overline{D}_k \}$ is also an
arithmetic $\RR$-Cartier divisor of $C^0$-type (cf. \cite[Lemma~9.1.2]{MoArZariski}).

\renewcommand{\theTheorem}{\arabic{section}.\arabic{Theorem}}
\renewcommand{\theClaim}{\arabic{section}.\arabic{Theorem}.\arabic{Claim}}
\renewcommand{\theequation}{\arabic{section}.\arabic{Theorem}.\arabic{Claim}}

\section{Relative Zariski decomposition of arithmetic divisors}
\label{sec:relative:Zariski:decomp}
We assume that $X$ is regular and $d = 1$.
The Stein factorization $X \to \Spec(O_K)$ of $X \to \Spec(\ZZ)$ is denoted by $\pi$.
Let $\overline{D} = (D, g)$ be an arithmetic $\RR$-Cartier divisor of $C^0$-type on $X$.
We say $\overline{D}$ is {\em relatively nef} if 
$c_1(\overline{D})$ is a positive current and $\adeg(\rest{\overline{D}}{C}) \geq 0$ for all
vertical reduced and irreducible $1$-dimensional closed subschemes $C$ on $X$.
We set
\[
\Upsilon_{rel}(\overline{D}) := \left\{ \overline{M} \ \left| \ \begin{array}{l}
\text{$\overline{M}$ is an arithmetic $\RR$-Cartier divisor of $C^0$-type} \\
\text{such that $\overline{M}$ is relatively nef and
$\overline{M} \leq \overline{D}$}
\end{array}
\right\}\right..
\]

\begin{Theorem}[Relative Zariski decomposition]
\label{thm:relative:Zariski:decomp}
If $\deg(D_K) \geq 0$, then
there is the greatest element $\overline{Q}$ of $\Upsilon_{rel}(\overline{D})$,
that is, $\overline{Q} \in \Upsilon_{rel}(\overline{D})$ and $\overline{M} \leq \overline{Q}$ for all
$\overline{M} \in \Upsilon_{rel}(\overline{D})$. 
Moreover, if we set $\overline{N} := \overline{D} - \overline{Q}$,
then $\overline{Q}$ and $\overline{N}$ satisfy the following properties:
\begin{enumerate}
\renewcommand{\labelenumi}{(\alph{enumi})}
\item
$N$ is vertical.

\item
$\adeg(\overline{Q} \cdot \overline{N}) = 0$.

\item
For any $P \in \Spec(O_K)$, $\pi^{-1}(P)_{red} \not\subseteq \Supp(N)$.

\item
The natural homomorphism $H^0(X, nQ) \to H^0(X, nD)$ is bijective and 
$\Vert\cdot\Vert_{n\overline{D}} = \Vert\cdot\Vert_{n\overline{Q}}$ for each $n \geq 0$.

\item
$\acvol(\overline{Q}) = \acvol(\overline{D})$.
\end{enumerate}
\end{Theorem}

Before staring the proof of Theorem~\ref{thm:relative:Zariski:decomp}, we need several preparations.
Let $D$ be an $\RR$-Cartier divisor on $X$. 
We say $D$ is {\em $\pi$-nef} if $\deg(\rest{D}{C}) \geq 0$ for all  
vertical reduced and irreducible $1$-dimensional closed subschemes $C$ on $X$.
First let us consider the relative Zariski decomposition on finite places.

\begin{Lemma}
\label{lem:vertical:Zariski:decomp}
Let $D$ be an $\RR$-Cartier divisor on $X$ and let
$\Sigma(D)$ be the set of all $\RR$-Cartier divisors $M$ on $X$ such that
$M$ is $\pi$-nef and $M \leq D$.
If $\deg(D_K) \geq 0$,
then there is the greatest element $Q$ of $\Sigma(D)$, that is,
$Q \in \Sigma(D)$ and $M \leq Q$ for all $M \in \Sigma(D)$.
Moreover, if we set $N := D - Q$, then $Q$ and $N$ satisfy the following properties:
\begin{enumerate}
\renewcommand{\labelenumi}{(\alph{enumi})}
\item
$N$ is vertical.

\item
$\deg(\rest{Q}{C}) = 0$ for all reduced and irreducible $1$-dimensional closed subschemes $C$ in $\Supp(N)$.

\item
For any $P \in \Spec(O_K)$, $\pi^{-1}(P)_{red} \not\subseteq \Supp(N)$.

\item
The natural homomorphism $H^0(X, nQ) \to H^0(X, nD)$ is bijective for each $n \geq 0$.
\end{enumerate}
\end{Lemma}

\begin{proof}
Let us begin with following claim:

\begin{Claim}
\label{claim:lem:vertical:Zariski:decomp:01}
$\Sigma(D) \not= \emptyset$.
\end{Claim}

\begin{proof}
First we assume that $\deg(D_K) = 0$.
Then, by using Zariski's lemma (cf. \cite[Lemma~1.1.4]{MoD}),
we can find a vertical and effective $\RR$-Cartier divisor $E$ such that
$\deg(\rest{(D-E)}{C}) = 0$ for all vertical reduced and irreducible $1$-dimensional closed subschemes $C$ on $X$, and hence
$\Sigma(D) \not= \emptyset$.

Next we assume that $\deg(D_K) > 0$.
Let $A$ be an ample Cartier divisor on $X$.
As $\deg(D_K) > 0$, $H^0(X_K, mD_K - A_K) \not= \{ 0 \}$ for some positive integer $m$, and hence
$H^0(X, mD - A) \not= \{ 0 \}$.
Thus,
there is $\phi \in \Rat(X)^{\times}$ such that
$m D - A + (\phi) \geq 0$, that is, $D \geq (1/m)(A - (\phi))$, as required.
\end{proof}

\begin{Claim}
\label{claim:lem:vertical:Zariski:decomp:02}
If $L_1, \ldots, L_k$ are $\pi$-nef $\RR$-Cartier divisors, then
$\max \{ L_1, \ldots, L_k \}$ is also $\pi$-nef
\rom{(}cf. Conventions and terminology~\rom{\ref{CT:max:arith:div}}\rom{)}.
\end{Claim}

\begin{proof}
We set $L'_i := \max \{ L_1, \ldots, L_k \} - L_i$ for each $i$.
Let $C$ be a vertical reduced and irreducible $1$-dimensional closed subscheme on $X$. Then there is $i$ such that
$C \not\subseteq \Supp(L'_i)$.
As $L'_i$ is effective, we have $\deg(\rest{L'_i}{C}) \geq 0$, so that
\[
\deg(\rest{\max \{ L_1, \ldots, L_k \}}{C}) = \deg(\rest{L_i}{C}) + \deg(\rest{L'_i}{C}) \geq 0.
\]
\end{proof}

For a reduced and irreducible $1$-dimensional closed subscheme $C$ on $X$, we set
\[
q_C := \sup \{ \mult_C(M) \mid M \in \Sigma(D) \},
\]
which exists in $\RR$ because $\mult_C(M) \leq \mult_C(D)$ for all $M \in \Sigma(D)$.
We fix $M_0 \in \Sigma(D)$.

\begin{Claim}
\label{claim:lem:vertical:Zariski:decomp:03}
There is a sequence $\{ M_n \}_{n=1}^{\infty}$ of $\RR$-Cartier divisors in $\Sigma(D)$ such that
$M_0 \leq M_n$ for all $n \geq 1$ and $\lim_{n\to\infty} \mult_{C}(M_n) = q_C$ for all reduced and
irreducible $1$-dimensional closed subschemes $C$ in $\Supp(D) \cup \Supp(M_0)$.
\end{Claim}

\begin{proof}
For each reduced and irreducible $1$-dimensional closed subscheme $C$ in $\Supp(D) \cup \Supp(M_0)$, 
there is a sequence $\{ M_{C, n} \}_{n=1}^{\infty}$ in $\Sigma(D)$ such that
\[
\lim_{n\to\infty} \mult_{C}(M_{C,n}) = q_{C}.
\]
If we set 
\[
M_n = \max \left( \{ M_{C, n} \}_{C \subseteq
\Supp(D) \cup \Supp(M_0)} \cup \{ M_0 \} \right),
\]
then
$M_0 \leq M_n$ and $M_n \in \Sigma(D)$ by Claim~\ref{claim:lem:vertical:Zariski:decomp:02}.
Moreover, as
\[
\mult_{C}(M_{C, n}) \leq \mult_{C}(M_n) \leq q_{C},
\]
$\lim_{n\to\infty} \mult_{C}(M_n) = q_{C}$.
\end{proof}

Since $\max \{ M_0, M \} \in \Sigma(D)$ for all $M \in \Sigma(D)$ by Claim~\ref{claim:lem:vertical:Zariski:decomp:02},
we have 
\[
\mult_C(M_0) \leq q_C \leq \mult_C(D).
\]
In particular, if $C \not\subseteq \Supp(D) \cup \Supp(M_0)$, then $q_C = 0$, so that we can set
$Q := \sum_C q_C C$.

\begin{Claim}
\label{claim:lem:vertical:Zariski:decomp:04}
$Q$ is the greatest element $Q$ in $\Sigma(D)$, that is, $Q \in \Sigma(D)$ and
$M \leq Q$ for all $M \in \Sigma(D)$.
\end{Claim}

\begin{proof}
By Claim~\ref{claim:lem:vertical:Zariski:decomp:03}, 
we can see that $Q \in \Sigma(D)$, so that the assertion follows.
\end{proof}

We need to check the properties (a) -- (d).

(a) We choose effective $\RR$-Cartier divisors $N_1$ and $N_2$ such that
$N = N_1 + N_2$, $N_1$ is horizontal and $N_2$ is vertical.
If $N_1 \not= 0$, then $Q \lneqq Q + N_1 \leq D$ and $Q + N_1$ is $\pi$-nef,
so that we have $N_1 = 0$, that is, $N$ is vertical.

(b) Let $C$ be a  vertical reduced and irreducible $1$-dimensional closed subscheme in $\Supp(N)$.
If $\deg(\rest{Q}{C}) > 0$, then $Q + \epsilon C$ is $\pi$-nef and $Q + \epsilon C \leq D$
for a sufficiently small $\epsilon > 0$, and
hence $\deg(\rest{Q}{C}) = 0$.

(c) 
We assume the contrary.
Then we can find $\delta > 0$ such that $\delta \pi^{-1}(P) \leq N$, so that
$Q \lneqq Q + \delta \pi^{-1}(P) \leq D$ and $Q + \delta \pi^{-1}(P)$ is $\pi$-nef.
This is a contradiction.

(d) 
It is sufficient to see that if $\phi \in \Gamma^{\times}(X, nD)$, then
$\phi \in \Gamma^{\times}(X, nQ)$.
Since $(-1/n)(\phi) \in \Sigma(D)$, we have $(-1/n)(\phi) \leq Q$, that is,
$nQ + (\phi) \geq 0$. Therefore $\phi \in \Gamma^{\times}(X, nQ)$.
\end{proof}

Moreover, we need the following lemma.

\begin{Lemma}
\label{lem:riemann:surface:norm}
Let $S$ be a connected compact Riemann surface and let $D$ be an $\RR$-divisor on $S$ with $\deg(D) \geq 0$.
Let $g$ be a $D$-Green function of $C^0$-type on $S$ and let $G(D,g)$ be the set of all $D$-Green functions $h$ 
of $C^0$-type on $S$ such that
$c_1(D, h)$ is a positive current and $h \leq g\ \aew$. 
Then there is the greatest element $q$ of $G(D,g)$, that is,
$q \in G(D,g)$ and $h \leq q\ \aew$ for all $h \in G(D,g)$.
Moreover, $q$ has the following property:
\begin{enumerate}
\renewcommand{\labelenumi}{(\arabic{enumi})}
\item
$\Vert \phi \Vert_{ng} = \Vert \phi \Vert_{nq}$ for all $\phi \in H^0(S, nD)$ and $n \geq 0$.

\item
${\displaystyle \int_S (g - q) c_1(D, q) = 0}$.
\end{enumerate}
\end{Lemma}

\begin{proof}
The existence of $q$ follows from \cite[Theorem~1.4]{BD} or \cite[Theorem~4.6]{MoArZariski}.
We need to check the properties (1) and (2).

(1) Clearly $\Vert \phi \Vert_{nq} \geq \Vert \phi \Vert_{ng}$ because $q \leq g\ \aew$.
Let us consider the converse inequality. We may assume that $\phi \not= 0$.
We set
\[
q' := \max \left\{ q, \frac{1}{n}\log (\vert \phi \vert^2/\Vert \phi \Vert_{ng}^2 ) \right\}.
\]
Since $D \geq (-1/n)(\phi)$ and $(1/n)\log (\vert \phi \vert^2/\Vert \phi \Vert_{ng}^2 )$ is
a $(-1/n)(\phi)$-Green function of $C^{\infty}$-type with the first Chern form zero,
by \cite[Lemma~9.1.1]{MoArZariski}, $q'$ is a $D$-Green function of $C^0$-type such that
$c_1(D, q')$ is a positive current.
Note that $\Vert \phi \Vert_{ng}^2 \geq \vert \phi \vert^2 \exp(-ng)\ \aew$, that is,
\[
g \geq (1/n)\log (\vert \phi \vert^2/\Vert \phi \Vert_{ng}^2 )\ \aew,
\]
and hence $q' \in G(D,g)$. Therefore, as $q' \geq q\ \aew$, we have $q = q'\ \aew$, so that
$q \geq (1/n)\log (\vert \phi \vert^2/\Vert \phi \Vert_{ng}^2 )\ \aew$, that is,
$\Vert \phi \Vert_{ng}^2 \geq \vert \phi \vert^2 \exp(-nq)\ \aew$,
which implies $\Vert \phi \Vert_{ng} \geq \Vert \phi \Vert_{nq}$.

(2) If $\deg(D) = 0$, then the assertion is obvious because $c_1(D, q) = 0$, so that we assume that $\deg(D) > 0$.
First we consider the case where $g$ is of $C^{\infty}$-type. 
We set $\alpha := c_1(D, g)$ and
\[
\varphi := \sup \left\{ \psi \mid \text{$\psi$ is an $\alpha$-plurisubharmonic function on $S$ and $\psi \leq 0$} \right\}
\]
(cf. \cite{BD}).
Then, by \cite[Proposition~4.3]{MoArZariski}, $q = g + \varphi\ \aew$. In particular, $\varphi$ is continuous because $g$ and $q$ are of $C^{0}$-type.
If we set 
$D = \{ x \in S \mid \varphi(x) = 0 \}$,
then,
by \cite[Corollary~2.5]{BD}, $c_1(D, q) = \bold{1}_D \alpha$, where $\bold{1}_D$ is the indicator function of $D$.
Thus
\[
\int_S (g - q) c_1(D, q) = 0.
\]

Next  we  consider a general case.
Let $g'$ be a $D$-Green function of $C^{\infty}$-type.
We set $g = g' + u\ \aew$ for some continuous function $u$ on $S$.
By using the Stone-Weierstrass theorem, we can find a sequence $\{ u_n \}$ of $C^{\infty}$-functions on $S$ such that
$\lim_{n\to\infty} \Vert u_n - u \Vert_{\sup} = 0$.
We set $g_n := g' + u_n$.
Let $q_n$ be the greatest element of $G(D, g_n)$.
As 
\[
g - \Vert u_n - u \Vert_{\sup} \leq g_n \leq g + \Vert u_n - u \Vert_{\sup}\ \aew,
\]
we can see $q - \Vert u_n - u \Vert_{\sup} \leq q_n \leq q + \Vert u_n - u \Vert_{\sup}\ \aew$.
Thus, if we set $q_n = g' + v_n\ \aew$ and $q = g' + v\ \aew$ for some continuous functions $v_n$ and $v$ on $S$,  
then
$\lim_{n\to\infty} \Vert v_n - v \Vert_{\sup} = 0$.
Moreover, by using the previous observation,
\begin{align*}
0 & = \int_S (g_n - q_n) c_1(D, q_n) = \int_{S} (u_n - v_n) c_1(D, q_n).
\end{align*}
Since $c_1(D, q_n) = c_1(D, g') + dd^c([v_n]) \geq 0$, by using \cite[Corollary~3.6]{Dem} or
\cite[Lemma~1.2.1]{MoD}, we can see
that $c_1(D, q_n)$ converges weakly to $c_1(D, q)$ as functionals on $C^0(S)$.
In particular, there is a constant $C$ such that 
$\int_S c_1(D, q_n) \leq C$
for all $n$.
Thus
\begin{multline*}
\left\vert \int_{S} (u_n - v_n) c_1(D, q_n) - \int_{S} (u - v) c_1(D, q) \right\vert \\
\leq
\left\vert \int_{S} (u_n - v_n) c_1(D, q_n) - \int_{S} (u - v) c_1(D, q_n) \right\vert \qquad\qquad\qquad\qquad\qquad \\
\qquad\qquad \qquad+ \left\vert \int_{S} (u - v) c_1(D, q_n) - \int_{S} (u - v) c_1(D, q) \right\vert \\
\leq \Vert (u-v) - (u_n - v_n) \Vert_{\sup}C + \left\vert \int_{S} (u - v) c_1(D, q_n) - 
\int_{S} (u - v) c_1(D, q) \right\vert.
\end{multline*}
Therefore,
\[
\lim_{n\to\infty} \int_{S} (u_n - v_n) c_1(D, q_n) = \int_{S} (u - v) c_1(D, q),
\]
and hence the assertion follows.
\end{proof}

\begin{proof}[Proof of Theorem~\rom{\ref{thm:relative:Zariski:decomp}}]
Let us start the proof of Theorem~\ref{thm:relative:Zariski:decomp}.
First we consider the existence of the greatest element of $\Upsilon_{rel}(\overline{D})$.
By Lemma~\ref{lem:vertical:Zariski:decomp},
there is the greatest element $Q$ of $\Sigma(D)$.
Note that $D - Q$ is vertical.
On the other hand,
let $G(\overline{D})$ be the set of all $D$-Green functions $h$ of $C^0$-type such that
$c_1(D, h)$ is a positive current and $h \leq g\ \aew$.
By Lemma~\ref{lem:riemann:surface:norm},
there is the greatest element $q$ of $G(\overline{D})$, that is,
$q \in G(\overline{D})$ and $h \leq q\ \aew$ for all $h \in G(\overline{D})$.
Let us see that $q$ is $F_{\infty}$-invariant. For this purpose, it is sufficient to see that
$F_{\infty}^*(q) \in G(\overline{D})$ and  $h \leq F_{\infty}^*(q)\ \aew$ for all $h \in G(\overline{D})$.
The first assertion follows from \cite[Lemma~5.1.2]{MoArZariski}.
Let us see the second assertion. Since $F_{\infty}^*(h) \in G(\overline{D})$ by \cite[Lemma~5.1.2]{MoArZariski},
$F_{\infty}^*(h) \leq q\ \aew$, and hence $h \leq F_{\infty}^*(q)\ \aew$.
Here we set $\overline{Q} := (Q, q)$.
Clearly $\overline{Q} \in \Upsilon_{rel}(\overline{D})$.
Moreover, for $\overline{M} \in \Upsilon_{rel}(\overline{D})$,
$(M', h') := \max \{ \overline{Q}, \overline{M} \} \in \Upsilon_{rel}(\overline{D})$
by
Claim~\ref{claim:lem:vertical:Zariski:decomp:02} and \cite[Lemma~9.1.1]{MoArZariski}
(for the definition of $\max \{ \overline{Q}, \overline{M} \}$,
see Conventions and terminology~\ref{CT:max:arith:div}).
In particular, $M' \in \Sigma(D)$ and $h' \in G(\overline{D})$, and hence
$(M', h') = \overline{Q}$, that is,
$\overline{M} \leq \overline{Q}$, as required.

\medskip
Finally let us see (a) --- (e).
As $Q$ is the greatest element of $\Sigma(D)$,
(a), (c) and the first assertion of (d) follow from Lemma~\ref{lem:vertical:Zariski:decomp}.
The second assertion of (d) follows from (1) in Lemma~\ref{lem:riemann:surface:norm}.
The property (e) is a consequence of (d).
Finally we consider (b).
If we set $\overline{N} = (N, k)$, then
$\adeg(\overline{Q} \cdot (N, 0)) = 0$ by (b) in Lemma~\ref{lem:vertical:Zariski:decomp},
and $\adeg(\overline{Q} \cdot (0, k)) = 0$ by (2) in Lemma~\ref{lem:riemann:surface:norm},
and hence $\adeg(\overline{Q} \cdot \overline{N}) = 0$.
\end{proof}

\section{Generalized Hodge index theorem for $\acvol$}
\label{sec:GHIT:acvol}

In this section, we consider a refinement of the generalized Hodge index theorem
on an arithmetic surface, that is, the case where $d=1$.
As in Conventions and terminology~\ref{CT:Div:Nef},
an arithmetic $\RR$-Cartier divisor $\overline{D}$ of $C^0$-type on $X$
is said to be {\em integrable} if
$\overline{D} = \overline{P} - \overline{Q}$ for some
nef arithmetic $\RR$-Cartier divisors $\overline{P}$ and $\overline{Q}$ of $C^0$-type.

\begin{Theorem}
\label{thm:cvol:self:ineq}
Let $\overline{D}$ be an integrable arithmetic $\RR$-Cartier divisor of $C^0$-type on $X$
such that $\deg(D_K) \geq 0$.
Then $\adeg (\overline{D}^2) \leq \acvol(\overline{D})$ and the equality holds if and only if
$\overline{D}$ is relatively nef.
In particular, $\adeg (\overline{D}^2) \leq \avol(\overline{D})$.
\end{Theorem}

\begin{proof}
Let $\mu : X' \to X$ be a desingularization of $X$ (cf. \cite{Lip}).
Then $\adeg(\overline{D}^2) = \adeg(\mu^*(\overline{D})^2)$ and
$\acvol(\overline{D}) = \acvol(\mu^*(\overline{D}))$.
Moreover, $\overline{D}$ is relatively nef if and only if $\mu^*(\overline{D})$ is relatively nef.
Therefore we may assume that $X$ is regular.

\begin{Claim}
\label{claim:thm:cvol:self:ineq:01}
If $\overline{D}$ is relatively nef,
then $\adeg (\overline{D}^2) = \acvol(\overline{D})$.
\end{Claim}

\begin{proof}
We divide the proof into five steps:

{\bf Step 1} (the case where $\overline{D}$ is an arithmetic $\QQ$-Cartier divisor of $C^{\infty}$-type and
$c_1(\overline{D})$ is a semi-positive form) :
In this case, the assertion follows from Ikoma \cite[Theorem~3.5.1]{Ikoma}.

{\bf Step 2} (the case where
$\overline{D}$ is of $C^{\infty}$-type,
$c_1(\overline{D})$ is a positive form and $\adeg(\rest{\overline{D}}{C}) > 0$ for all vertical reduced and irreducible 
$1$-dimensional closed subschemes $C$) :
We choose
arithmetic Cartier divisors $\overline{D}_1, \ldots, \overline{D}_l$
of $C^{\infty}$-type and real numbers $a_1, \ldots, a_l$ such that
$\overline{D} = a_1 \overline{D}_1 + \cdots + a_l \overline{D}_l$.
Then there is a positive number $\delta_0$ such that
$c_1(b_1 \overline{D}_1 + \cdots + b_l \overline{D}_l)$ is a positive form for all
$b_1, \ldots, b_l \in \QQ$ with $\vert b_i - a_i \vert \leq \delta_0$ ($\forall i=1, \ldots, l$).
Let $C$ be a smooth fiber of $X \to \Spec(O_K)$ over $P$.
Then, for $b_1, \ldots, b_l \in \QQ$ with $\vert b_i - a_i \vert \leq \delta_0$ ($\forall i=1, \ldots, l$),
\[
\adeg \left( \rest{(b_1 \overline{D}_1 + \cdots + b_l \overline{D}_l)}{C} \right)
= \deg((b_1 D_1 + \cdots + b_l D_l)_K) \log \#(O_K/P) > 0.
\]
Let $C_1, \ldots, C_r$ be all irreducible components of singular fibers of $X \to \Spec(O_K)$.
Then, for each $j=1, \ldots, r$,
there is a positive number $\delta_j$ such that
\[
\adeg \left( \rest{(b_1 \overline{D}_1 + \cdots + b_l \overline{D}_l)}{C_j} \right) > 0
\]
for all $b_1, \ldots, b_l \in \QQ$ with $\vert b_i - a_i \vert \leq \delta_j$ ($\forall i=1, \ldots, l$).
Therefore, if we set $\delta = \min \{ \delta_0, \delta_1, \ldots, \delta_r \}$, then,
for $b_1, \ldots, b_l \in \QQ$ with $\vert b_i - a_i \vert \leq \delta$ ($\forall i=1, \ldots, l$),
\[
c_1(b_1 \overline{D}_1 + \cdots + b_l \overline{D}_l)
\]
is a positive form and
$\adeg \left( \rest{(b_1 \overline{D}_1 + \cdots + b_l \overline{D}_l)}{C} \right) > 0$
for all vertical reduced and irreducible $1$-dimensional closed subschemes $C$ on $X$, and hence
\[
\adeg ((b_1 \overline{D}_1 + \cdots + b_l \overline{D}_l)^2) = \acvol(b_1 \overline{D}_1 + \cdots + b_l \overline{D}_l)
\]
by Step~1.
Thus the assertion follows by the continuity of $\acvol$ due to Ikoma \cite[Corollary~3.4.4]{Ikoma}.

{\bf Step 3} (the case where $\overline{D}$ is of $C^{\infty}$-type and
$c_1(\overline{D})$ is a semi-positive form) :
Let $\overline{A}$ be an ample arithmetic Cartier divisor of $C^{\infty}$-type on $X$.
Then, for any positive $\epsilon$, $c_1(\overline{D} + \epsilon \overline{A})$ is
a positive form and $\adeg(\rest{(\overline{D} + \epsilon \overline{A})}{C}) > 0$ 
for all vertical reduced and irreducible $1$-dimensional closed subschemes $C$ on $X$, so that, by Step~2,
\[
\adeg((\overline{D} + \epsilon  \overline{A})^2) =
\acvol(\overline{D} + \epsilon  \overline{A}).
\]
Therefore the assertion follows by virtue of the continuity of $\acvol$.

{\bf Step 4} (the case where $\deg(D_K) > 0$) :
Let $h$ be a $D$-Green function of $C^{\infty}$-type such that
$c_1(D, h)$ is a positive form.
Then there is a continuous function $\phi$ on $X(\CC)$ such that
$\overline{D} = (D, h + \phi) $, and hence
$c_1(D, h) + dd^c([\phi]) \geq 0$.
Thus, by \cite[Lemma~4.2]{MoArZariski}, 
there is a sequence $\{ \phi_n \}_{n=1}^{\infty}$ 
of $F_{\infty}$-invariant $C^{\infty}$-functions on $X(\CC)$ with the following properties:
\begin{enumerate}
\renewcommand{\labelenumi}{(\alph{enumi})}
\item
$\lim_{n\to\infty} \Vert \phi_n - \phi \Vert_{\sup} = 0$.

\item
If we set $\overline{D}_n = (D, h + \phi_n)$, then
$c_1(\overline{D}_n)$ is a semipositive form.
\end{enumerate}
Then, by Step~3, 
$\adeg(\overline{D}_n^2) =
\acvol(\overline{D}_n)$ for all $n$.
Note that $\lim_{n\to\infty} \acvol(\overline{D}_n) = \acvol(\overline{D})$
by using the continuity of $\acvol$.
Moreover, by \cite[Lemma~1.2.1]{MoD},
\[
\lim_{n\to\infty} \adeg(\overline{D}_n^2) = \adeg(\overline{D}^2),
\]
as required.

{\bf Step 5} (general case) :
Finally we prove the assertion of the claim.
As before,  let $\overline{A}$ be an ample arithmetic Cartier divisor 
of $C^{\infty}$-type on $X$.
Then, for any positive number $\epsilon$, $\deg(D_K + \epsilon A_K) > 0$.
Thus, in the same way as Step~3, the assertion follows from Step~4.
\end{proof}

Let us go back to the proof of the theorem.
Let $\overline{Q}$ be the greatest element of $\Upsilon_{rel}(\overline{D})$ 
(cf. Theorem~\ref{thm:relative:Zariski:decomp}) and $\overline{N} := \overline{D} - \overline{Q}$.
Then, by using Claim~\ref{claim:thm:cvol:self:ineq:01} and
the properties (b) and (e) in Theorem~\ref{thm:relative:Zariski:decomp},
\[
\acvol(\overline{D}) - \adeg(\overline{D}^2) = \acvol(\overline{Q}) - \adeg(\overline{D}^2) =
\adeg(\overline{Q}^2) - \adeg(\overline{D}^2) = - \adeg(\overline{N}^2).
\]
On the other hand, if we set $\overline{N} = (N, k)$, then
\[
\adeg(\overline{N}^2) = \adeg((N, 0)^2) + \frac{1}{2} \int_{X(\CC)} k dd^c(k)
\]
because $N$ is vertical.
By (c) in Theorem~\ref{thm:relative:Zariski:decomp} together with
Zariski's lemma, $\adeg((N, 0)^2) \leq 0$ and the equality holds if and only if $N = 0$.
Moreover, by \cite[Proposition~1.2.3 and Proposition~2.1.1]{MoD},
\[
\int_{X(\CC)} k dd^c(k) \leq 0
\]
and the equality holds if and only if $k$ is locally constant.
Thus $\adeg(\overline{N}^2) \leq 0$, that is,
$\acvol(\overline{D}) \geq \adeg(\overline{D}^2)$.
Moreover, if $\overline{D}$ is relatively nef, then
$\acvol(\overline{D}) = \adeg(\overline{D}^2)$ by Claim~\ref{claim:thm:cvol:self:ineq:01}.
Conversely, if $\acvol(\overline{D}) = \adeg(\overline{D}^2)$, that is, $\adeg(\overline{N}^2) = 0$,
then $N=0$ and $k$ is locally constant, and hence $\overline{D} = \overline{Q} + (0, k)$ is
relatively nef.
\end{proof}

As a corollary of the above theorem, 
we have the following:

\begin{Corollary}
\label{cor:properties:relative:Zariski:decomp}
We assume that $X$ is regular.
The following are equivalent:
\begin{enumerate}
\renewcommand{\labelenumi}{(\arabic{enumi})}
\item
$\overline{Q}$ is the greatest element of $\Upsilon_{rel}(\overline{D})$.

\item
$\overline{Q}$ is an element of $\Upsilon_{rel}(\overline{D})$ with the following properties:
\begin{enumerate}
\renewcommand{\labelenumii}{(\roman{enumii})}
\item
$D - Q$ is vertical.

\item
$\adeg(\overline{Q} \cdot \overline{B}) = 0$ and
$\adeg(\overline{B}^2) < 0$ for all 
integrable arithmetic $\RR$-Cartier divisors $\overline{B}$ of $C^0$-type with
$(0,0) \lneqq \overline{B} \leq \overline{D} - \overline{Q}$.
\end{enumerate}
\end{enumerate}
\end{Corollary}

\begin{proof}
First, let us see the following claim:

\begin{Claim}
\label{claim:cor:properties:relative:Zariski:decomp:01}
Let $\overline{D}_1$ and $\overline{D}_2$ be arithmetic $\RR$-Cartier divisors of $C^0$-type on $X$
such that $\overline{D}_1 \leq \overline{D}_2$. If the natural map $H^0(X, nD_1) \to H^0(X, nD_2)$ is
bijective for each $n \geq 0$, then $\acvol(\overline{D}_1) \leq \acvol(\overline{D}_2)$,
\end{Claim}

\begin{proof}
This is obvious because $\Vert\cdot\Vert_{n\overline{D}_1} \geq \Vert\cdot\Vert_{n\overline{D}_2}$.
\end{proof}

(1)  $\Longrightarrow$ (2) :
By the property (a) in Theorem~\ref{thm:relative:Zariski:decomp}, $D - Q$ is vertical.
For $0 < \epsilon \leq 1$, we set
$\overline{D}_{\epsilon} = \overline{Q} + \epsilon \overline{B}$.
Then $\overline{D}_{\epsilon}$ is integrable
and $\acvol(\overline{D}_{\epsilon}) = \acvol(\overline{Q})$
because
\[
\acvol(\overline{Q}) \leq \acvol(\overline{D}_{\epsilon}) \leq \acvol(\overline{D})
\quad\text{and}\quad
\acvol(\overline{Q}) = \acvol(\overline{D})
\]
by Claim~\ref{claim:cor:properties:relative:Zariski:decomp:01} and the properties (d) and (e) in
Theorem~\ref{thm:relative:Zariski:decomp}.
Thus, by using Theorem~\ref{thm:cvol:self:ineq},
\[
\adeg(\overline{D}_{\epsilon}^2) \leq \acvol(\overline{D}_{\epsilon}) = \acvol(\overline{Q}) = \adeg(\overline{Q}^2),
\]
which implies $2\adeg(\overline{Q} \cdot \overline{B}) + \epsilon \adeg(\overline{B}^2) \leq 0$.
In particular, $\adeg(\overline{Q} \cdot \overline{B}) \leq 0$. On the other hand, 
as $B$ is vertical,
\[
\adeg(\overline{Q} \cdot \overline{B}) = \adeg(\overline{Q} \cdot (B, 0)) + \frac{1}{2} \int_{X(\CC)} c_1(\overline{Q}) b 
\geq 0
\]
where $\overline{B} = (B, b)$. Therefore, $\adeg(\overline{Q} \cdot \overline{B}) = 0$ and
$\adeg(\overline{B}^2) \leq 0$.
Here we assume that $\adeg(\overline{B}^2) = 0$.
Note that
\[
\adeg(\overline{B}^2) = \adeg((B,0)^2) + \frac{1}{2} \int_{X(\CC)} b dd^c(b).
\]
Thus, by using the property (c) in Theorem~\ref{thm:relative:Zariski:decomp}, Zariski's lemma
and \cite[Proposition~1.2.3 and Proposition~2.1.1]{MoD},
$B = 0$ and $b$ is a locally constant function.
In particular, $\overline{Q} + \overline{B}$ is relatively nef and $\overline{Q} + \overline{B} \leq \overline{D}$,
so that $\overline{B} = 0$.

\medskip
(2)  $\Longrightarrow$ (1) :
Let $\overline{M}$ be an element of  $\Upsilon_{rel}(\overline{D})$. 
If
we set $\overline{A} := \max \{ \overline{Q}, \overline{M} \}$ 
(cf. Conventions and terminology~\ref{CT:max:arith:div})
and $\overline{B} = (B, b) := \overline{A} - \overline{Q}$,
then $\overline{B}$ is effective, $\overline{A} \leq \overline{D}$ and $\overline{A}$ is relatively nef 
by Claim~\ref{claim:lem:vertical:Zariski:decomp:02} and \cite[Lemma~9.1.2]{MoArZariski}. 
Moreover, 
\[
\overline{B} = \overline{A} - \overline{Q} \leq \overline{D} - \overline{Q}.
\]
If we assume $\overline{B} \gneqq (0,0)$, then,
by the property (ii), $\adeg(\overline{Q} \cdot \overline{B}) = 0$ and $\adeg(\overline{B}^2) < 0$.
On the other hand, 
as $\overline{A}$ is relatively nef, $\overline{B}$ is effective and $B$ is vertical
by the property (i),
\[
\adeg(\overline{B}^2) = \adeg(\overline{Q} + \overline{B} \cdot \overline{B }) = \adeg(\overline{A} \cdot \overline{B})
=\adeg(\overline{A} \cdot (B, 0)) + \frac{1}{2} \int_{X(\CC)} c_1(\overline{A}) b \geq 0,
\]
which is a contradiction, so that $\overline{B} = (0,0)$, that is,
$\overline{Q} = \overline{A}$, which means that $\overline{M} \leq \overline{Q}$, as required.
\end{proof}

\begin{Remark}
Let $\overline{D}$ be an integrable arithmetic $\RR$-Cartier divisor of $C^0$-type on $X$
with $\deg(D_K) > 0$.
For a positive number $\epsilon$, we set
\[
\alpha := \frac{\adeg(\overline{D}^2)}{[K : \QQ]\deg(D_K)} - 2 \epsilon.
\]
Then, as
$\adeg((\overline{D} - (0, \alpha))^2) = 2 \epsilon [K : \QQ]\deg(D_K) > 0$, by Theorem~\ref{thm:cvol:self:ineq},
there is 
\[
\phi \in \aH(X, n(D - (0,\alpha))) \setminus \{ 0 \}
\]
for some $n > 0$.
Note that $\Vert \phi \Vert_{n(\overline{D}-(0,\alpha))} = \Vert \phi \Vert_{n\overline{D}}\exp((n\alpha)/2)$, so that
\[
\phi \in H^0(X, nD) \setminus \{ 0\}\quad\text{and}\quad
\Vert \phi \Vert_{n\overline{D}} \leq \exp\left( -\frac{n \adeg(\overline{D}^2)}{2[K : \QQ]\deg(D_K)} + n\epsilon \right),
\]
which is nothing more than  Autissier's result \cite[Proposition~3.3.3]{AP}.
\end{Remark}

\section{Necessary condition for the equality $\avol = \acvol$}
This section is devoted to consider a necessary condition for the equality $\avol = \acvol$
as an application of the integral formulae due to Boucksom-Chen \cite{BC}.

First of all,
let us review Boucksom-Chen's integral formulae \cite{BC} in terms of 
arithmetic $\RR$-Cartier divisors. 
For details, see \cite[Section~1]{MoArLinB}.
We fix a monomial order $\precsim$ on $\ZZ^d_{\geq 0}$, that is,
$\precsim$ is a total ordering relation on $\ZZ^d_{\geq 0}$ with the following properties:
\begin{enumerate}
\renewcommand{\labelenumi}{(\alph{enumi})}
\item $(0,\ldots,0) \precsim A$ for all $A \in \ZZ^d_{\geq 0}$.

\item If $A \precsim B$ for $A, B \in \ZZ_{\geq 0}^d$, then $A + C \precsim B + C$ for all $C \in \ZZ_{\geq 0}^d$.

\end{enumerate}
The monomial order $\precsim$ on $\ZZ^d_{\geq 0}$ extends uniquely to a totally ordering relation $\precsim$ on $\ZZ^d$
such that $A + C \precsim B + C$ for all $A, B, C \in \ZZ^d$ with $A \precsim B$.
Indeed, for $A, B \in \ZZ^d$, we define $A \precsim B$ as follows:
\[
A \precsim B\quad\overset{\mathrm{def}}{\Longleftrightarrow}\quad
\text{there is $C \in \ZZ_{\geq 0}^d$ such that $A + C, B + C \in \ZZ_{\geq 0}^d$ and $A + C \precsim B + C$}.
\]
It is easy to see that this definition is well-defined and it yields the above extension.
Uniqueness is also obvious.

Let $z_P = (z_1, \ldots, z_d)$ be a local system of parameters of $\OO_{X_{\overline{K}}, P}$ for $P \in X(\overline{K})$.
Note that the completion $\widehat{\OO}_{X_{\overline{K}}, P}$ of 
$\OO_{X_{\overline{K}}, P}$ with respect to
the maximal ideal of $\OO_{X_{\overline{K}}, P}$ is naturally isomorphic to
$\overline{K}[\![z_1, \ldots, z_d]\!]$,
where $\overline{K}$ is an algebraic closure of $K$.
Thus, for $f \in \OO_{X_{\overline{K}}, P}$, we can put
\[
f = \sum_{(a_1, \ldots, a_d) \in \ZZ_{\geq 0}^d} c_{(a_1, \ldots, a_d)} z_1^{a_1} \cdots z_d^{a_d},\qquad(c_{(a_1, \ldots, a_d)} \in \overline{K}).
\]
We define $\ord_{z_P}^{\precsim}(f)$ to be
\[
\ord_{z_P}^{\precsim}(f) := \begin{cases}
\min\limits_{\precsim} \left\{ (a_1, \ldots, a_d) \mid c_{(a_1, \ldots, a_d)} \not= 0 \right\} & \text{if $f \not= 0$},\\
\infty & \text{otherwise},
\end{cases}
\]
which gives rise to a rank $d$ valuation, that is,
the following properties are satisfied:
\begin{enumerate}
\renewcommand{\labelenumi}{(\roman{enumi})}
\item
$\ord_{z_P}^{\precsim}(fg) = \ord_{z_P}^{\precsim}(f) + \ord_{z_P}^{\precsim}(g)$ for $f, g \in \OO_{X_{\overline{K}},P}$.

\item
$\min \left\{ \ord_{z_P}^{\precsim}(f), \ord_{z_P}^{\precsim}(g) \right\} \precsim \ord_{z_P}^{\precsim}(f + g)$ for $f, g \in \OO_{X_{\overline{K}},P}$.
\end{enumerate}
By the property (i), $\ord_{z_P}^{\precsim} : \OO_{X_{\overline{K}},P} \setminus \{ 0 \} \to \ZZ_{\geq 0}^d$ has the natural extension
\[
\ord_{z_P}^{\precsim} : \Rat(X_{\overline{K}})^{\times} \to \ZZ^d
\]
given by $\ord_{z_P}^{\precsim}(f/g) = \ord_{z_P}^{\precsim}(f) - \ord_{z_P}^{\precsim}(g)$.
Note that this extension also satisfies the same properties (i) and (ii) as before.
Since $\ord_{z_P}^{\precsim}(u) = (0,\ldots,0)$ for all $u \in \OO^{\times}_{X_{\overline{K}},P}$, 
$\ord_{z_P}^{\precsim}$ induces
$\Rat(X_{\overline{K}})^{\times}/\OO^{\times}_{X_{\overline{K}},P} \to \ZZ^d$. 
The composition of homomorphisms
\[
\Div(X_{\overline{K}}) \overset{\alpha_P}{\longrightarrow} \Rat^{\times}(X_{\overline{K}})/\OO^{\times}_{X_{\overline{K}},P} \overset{\ord_{z_P}^{\precsim}}{\longrightarrow} \ZZ^d
\]
is
denoted by $\mult_{z_P}^{\precsim}$, where 
$\alpha_P : \Div(X_{\overline{K}}) \to \Rat(X_{\overline{K}})^{\times}/\OO^{\times}_{X_{\overline{K}},P}$ is the natural
homomorphism. Moreover, the homomorphism $\mult_{z_P}^{\precsim} : \Div(X_{\overline{K}}) \to \ZZ^d$ gives rise to
the natural extension
$\Div(X_{\overline{K}}) \otimes_{\ZZ} \RR \to \RR^d$
over $\RR$.
By abuse of notation, the above extension is also denoted by $\mult_{z_P}^{\precsim}$.

Let $\overline{D} = (D, g)$ be an arithmetic $\RR$-Cartier divisor of $C^0$-type
(cf. Conventions and terminology~\ref{CT:arith:div}).
Let $V_{\bullet} =  \bigoplus_{m \geq 0} V_m$ be a graded subalgebra of $R(D_K) := \bigoplus_{m \geq 0} H^0(X_K, mD_K)$ over $K$.
The Okounkov body $\Delta(V_{\bullet})$ of $V_{\bullet}$ is defined by
the closed convex hull of
\[
\bigcup_{m > 0} \left\{ \mult_{z_P}^{\precsim}(D_{\overline{K}} + (1/m) (\phi)) \in \RR_{\geq 0}^d \mid \phi \in V_{m} \otimes_K \overline{K} \setminus \{ 0 \} \right\}.
\]
For $t \in \RR$, let $V_{\bullet}^t$ be a graded subalgebra of $V_{\bullet}$ given by
\[
V_{\bullet}^t := \bigoplus_{m \geq 0}  \left\langle V_m \cap \aH(X, m (\overline{D} + (0, -2t))) \right\rangle_K,
\]
where $\left\langle V_m \cap \aH(X, m (\overline{D} + (0, -2t))) \right\rangle_K$ means the subspace of $V_m$ generated by
$V_m \cap \aH(X, m (\overline{D} + (0, -2t)))$ over $K$.
Let $G_{(\overline{D};V_{\bullet})} : \Delta(V_{\bullet}) \to \RR \cup \{ -\infty \}$ 
be a function given by
\[
G_{(\overline{D};V_{\bullet})}(x) := \begin{cases}
\sup \left\{ t \in \RR \mid x \in \Delta(V_{\bullet}^t) \right\} & \text{if $x \in \Delta(V_{\bullet}^t)$ for some $t$}, \\
-\infty & \text{otherwise}.
\end{cases}
\]
Note that $G_{(\overline{D};V_{\bullet})}$ is an upper
semicontinuous concave function (cf. \cite[SubSection~1.3]{BC}). 
We define $\avol(\overline{D};V_{\bullet})$ and $\acvol(\overline{D};V_{\bullet})$ 
to be
\[
\begin{cases}
{\displaystyle \avol(\overline{D};V_{\bullet}) := \limsup_{m\to\infty} 
\frac{\# \log \left( V_m \cap \aH(X, m\overline{D}) \right)}{m^{d+1}/(d+1)!}}, \\
{\displaystyle  \acvol(\overline{D};V_{\bullet}) := \limsup_{m\to\infty} 
\frac{\hat{\chi} \left(V_m \cap H^0(X, mD), \Vert\cdot\Vert_{m\overline{D}} \right)}{m^{d+1}/(d+1)!}}. 
\end{cases}
\]
Moreover, for $\xi \in X_K$,
we define $\mu_{\QQ,\xi}(\overline{D}; V_{\bullet})$ as follows:
\begin{multline*}
\mu_{\QQ,\xi}(\overline{D}; V_{\bullet}) :=  \\
\begin{cases}
\inf \left\{ \mult_{\xi}\left(D + \frac{1}{m}(\phi)\right) \mid m \in \ZZ, \ \phi \in V_m \cap \aH(X, m\overline{D}) \setminus \{ 0 \} \right\} & 
\text{if $N(\overline{D}; V_{\bullet}) \not= \emptyset$}, \\
\infty & \text{otherwise},
\end{cases}
\end{multline*}
where $N(\overline{D}; V_{\bullet}) = \{ m \in \ZZ_{>0} \mid V_m \cap \aH(X, m\overline{D}) \not= \{ 0 \} \}$.
Note that $\avol(\overline{D};V_{\bullet}) = \avol(\overline{D})$,
$\acvol(\overline{D};V_{\bullet}) = \acvol(\overline{D})$ and $\mu_{\QQ,\xi}(\overline{D}; V_{\bullet}) = \mu_{\QQ,\xi}(\overline{D})$
if $V_m = H^0(X_K, mD_K)$ for $m \gg 0$
(cf. Conventions and terminology~\ref{CT:arith:div} and \ref{CT:volume}).
Let $\Theta(\overline{D}; V_{\bullet})$ be the closure of  
\[
\left\{ x \in \Delta(V_{\bullet}) \mid 
G_{(\overline{D};V_{\bullet})}(x) > 0 \right\}.
\]
If $V_{\bullet}$ contains an ample series (cf. \cite[SubSection~1.1]{MoArLinB}),
then, in the similar way as  \cite[Theorem~2.8]{BC} and \cite[Theorem~3.1]{BC}, 
we have the following integral formulae:
\renewcommand{\theequation}{\arabic{section}.\arabic{Theorem}}
\addtocounter{Theorem}{1}
\begin{equation}
\label{eqn:integral:formula:vol}
\avol(\overline{D};V_{\bullet}) = (d+1)! [K : \QQ] \int_{\Theta(\overline{D}; V_{\bullet})}
G_{(\overline{D};V_{\bullet})}(x) dx
\end{equation}
and
\addtocounter{Theorem}{1}
\begin{equation}
\label{eqn:integral:formula:chivol}
\acvol(\overline{D};V_{\bullet}) = (d+1)! [K : \QQ] \int_{\Delta(V_{\bullet})}
G_{(\overline{D};V_{\bullet})}(x) dx.
\end{equation}
\renewcommand{\theequation}{\arabic{section}.\arabic{Theorem}.\arabic{Claim}}
Note that the arguments in \cite{BC} work for an arbitrary monomial order.
Let $\nu : \RR^d \to \RR$ be a linear map.
If we give the monomial order $\prec_{\nu}$ on $\ZZ_{\geq 0}^d$ by the following rule:
\[
a \prec_{\nu} b \quad\overset{\mathrm{def}}{\Longleftrightarrow}\quad
\text{either $\nu(a) < \nu(b)$, or $\nu(a) = \nu(b)$ and $a \prec_{\mathrm{lex}} b$},
\]
then
$\nu(a) \leq \nu(b)$ for all $a, b \in \ZZ_{\geq 0}^d$ with $a \precsim_{\nu} b$.
Let us begin with the following lemma.

\begin{Lemma}
\label{lem:vol:asym:mu}
If $V_{\bullet}$ contains an ample series and $\avol(\overline{D};V_{\bullet}) > 0$, 
then we have the following:
\begin{enumerate}
\renewcommand{\labelenumi}{(\arabic{enumi})}
\item
$\Theta(\overline{D}; V_{\bullet}) = \Delta(V_{\bullet}^0) = \left\{ x \in \Delta(V_{\bullet}) \mid G_{(\overline{D};V_{\bullet})}(x) \geq 0 \right\}$.

\item
We assume that $\nu$ is given by $\nu(x_1, \ldots, x_d) = x_1 + \cdots + x_r$, where $1 \leq r \leq d$.
We further assume that 
the monomial order $\precsim$ satisfies
the property: $\nu(a) \leq \nu(b)$ for all $a, b \in \ZZ_{\geq 0}^d$ with $a \precsim b$.
Let $B$ is a reduced and irreducible subvariety of $X_{\overline{K}}$ such that $B$ is given by $z_1 = \cdots = z_r = 0$ around $P$.
Then $\mu_{\QQ, B}(\overline{D};V_{\bullet}) = \min \left\{ \nu(x) \mid x \in \Theta(\overline{D}; V_{\bullet}) \right\}$.
\end{enumerate}
\end{Lemma}

\begin{proof}
(1) Note that 
\[
\left\{ x \in \Delta(V_{\bullet}) \mid G_{(\overline{D};V_{\bullet})}(x) > 0 \right\} \subseteq \Delta(V_{\bullet}^0) \subseteq
\left\{ x \in \Delta(V_{\bullet}) \mid G_{(\overline{D};V_{\bullet})}(x) \geq 0 \right\}
\]
and $\left\{ x \in \Delta(V_{\bullet}) \mid G_{(\overline{D};V_{\bullet})}(x) \geq 0 \right\}$ is closed because
$G_{(\overline{D};V_{\bullet})}$ is upper semicontinuous.
Thus it is sufficient to show that $\left\{ x \in \Delta(V_{\bullet}) \mid G_{(\overline{D};V_{\bullet})}(x) \geq 0 \right\}
\subseteq \Theta(\overline{D}; V_{\bullet})$.
Let $x \in  \Delta(V_{\bullet})$ with
$G_{(\overline{D};V_{\bullet})}(x) \geq 0$. 
As
\[
\avol(\overline{D};V_{\bullet}) = (d+1)! [K : \QQ] \int_{\Theta(\overline{D}; V_{\bullet})} G_{(\overline{D};V_{\bullet})}(x) dx > 0
\]
by \eqref{eqn:integral:formula:vol},
we can choose $y \in  \Theta(\overline{D}; V_{\bullet})$ with
$G_{(\overline{D};V_{\bullet})}(y) > 0$.
Then
\[
G_{(\overline{D};V_{\bullet})}((1-t)x + ty) \geq (1-t)G_{(\overline{D};V_{\bullet})}(x) + t G_{(\overline{D};V_{\bullet})}(y)
\geq t G_{(\overline{D};V_{\bullet})}(y) >0
\]
for all $t \in \RR$ with $0 < t \leq 1$.
Thus $x \in \Theta(\overline{D}; V_{\bullet})$.

(2) First let us see the following claim:

\begin{Claim}
\label{claim:lem:vol:asym:mu:01}
For $L \in \Div(X)_{\RR}$, $\nu\left(\mult_{z_P}^{\precsim}(L)\right) = \mult_B(L)$.
\end{Claim}

\begin{proof}
It is sufficient to see that $\nu\left(\ord_{z_P}^{\precsim}(f)\right) = \ord_B(f)$ for
$f \in \OO_{X_{\overline{K}}} \setminus \{ 0 \}$. 
We set $f = \sum_{\beta \in \ZZ^{d}_{\geq 0}} c_{\beta} z^{\beta}$ and  $\alpha = \ord_{z_P}^{\precsim}(f)$.
Note that $\ord_B(f) = \min \{ \nu(\beta) \mid c_{\beta} \not= 0 \}$. 
Thus the assertion follows because $c_{\alpha} \not= 0$ and
$\nu(\alpha) \leq \nu(\beta)$ for $\beta \in \ZZ_{\geq 0}^d$ with $c_{\beta} \not= 0$.
\end{proof}

If we set 
\[
x_{\phi} = \mult_{z_P}^{\precsim}(D + (1/m)(\phi))
\]
for
$\phi \in V_m \cap \aH(X, m\overline{D}) \setminus \{ 0 \}$ and $m > 0$, 
then $G_{(\overline{D};V_{\bullet})}(x_{\phi}) \geq 0$ by the definition of $G_{(\overline{D};V_{\bullet})}$, and hence,
$x_{\phi} \in \Theta(\overline{D}; V_{\bullet})$ by (1). Therefore, by Claim~\ref{claim:lem:vol:asym:mu:01},
\[
\min \{ \nu(x) \mid x \in \Theta(\overline{D}; V_{\bullet}) \} \leq 
\nu(x_{\phi}) = \mult_{\xi}(D + (1/m)(\phi)),
\]
which implies 
$\min \{ \nu(x) \mid x \in \Theta(\overline{D}; V_{\bullet}) \} \leq \mu_{\QQ, B}(\overline{D};V_{\bullet})$.

\begin{Claim}
\label{claim:lem:vol:asym:mu:02}
\[
\mu_{\QQ, B}(\overline{D};V_{\bullet}) 
\leq \nu\left(\mult_{z_P}^{\precsim} \left( D + (1/m) \left( \sum_{\phi \in V_m \cap \aH(X, m\overline{D}) \setminus \{ 0 \}} c_{\phi} \phi \right) \right)\right),
\]
where $c_{\phi} \in \overline{K}$ and $\sum_{\phi \in V_m \cap \aH(X, m\overline{D}) \setminus \{ 0 \}} c_{\phi} \phi \not= 0$.
\end{Claim}

\begin{proof}
By the property (ii),
\[
\min_{\phi \in V_m \cap \aH(X, m\overline{D}) \setminus \{ 0 \}} \left\{ \ord_{z_P}^{\precsim}(\phi) \right\} \precsim
\ord_{z_P}^{\precsim}\left(\sum_{\phi \in V_m \cap \aH(X, m\overline{D}) \setminus \{ 0 \}} c_{\phi} \phi\right)
\]
on $\ZZ^d$, which yields
\[
\min_{\phi \in V_m \cap \aH(X, m\overline{D}) \setminus \{ 0 \}} \left\{ \nu\left(\ord_{z_P}^{\precsim}(\phi)\right) \right\} \leq
\nu\left(\ord_{z_P}^{\precsim}\left(\sum_{\phi \in V_m \cap \aH(X, m\overline{D}) \setminus \{ 0 \}} c_{\phi} \phi\right)\right),
\]
and hence
\begin{multline*}
\min_{\phi \in V_m \cap \aH(X, m\overline{D}) \setminus \{ 0 \}} \left\{ \nu\left(\mult_{z_P}^{\precsim}(D + (1/m)(\phi))\right) \right\} \\
\leq
\nu\left(\mult_{z_P}^{\precsim}\left(D + (1/m)\left(\sum_{\phi \in V_m \cap \aH(X, m\overline{D}) \setminus \{ 0 \}} c_{\phi} \phi\right)\right)\right).
\end{multline*}
Thus the claim follows by Claim~\ref{claim:lem:vol:asym:mu:01}.
\end{proof}

By the above claim together with (1),
\[
\Theta(\overline{D}; V_{\bullet}) = \Delta(V_{\bullet}^0) \subseteq \{ x \in \Delta(V_{\bullet}) \mid \mu_{\QQ, B}(\overline{D};V_{\bullet}) \leq \nu(x) \},
\]
which shows that $\min \{ \nu(x) \mid x \in \Theta(\overline{D}; V_{\bullet}) \} \geq \mu_{\QQ, B}(\overline{D}; V_{\bullet})$,
as required.
\end{proof}

The following theorem is the main result of this section.

\begin{Theorem}
\label{thm:vol:cvol:mu:zero}
If $V_{\bullet}$ contains an ample series, $\avol(\overline{D};V_{\bullet}) = \acvol(\overline{D};V_{\bullet}) > 0$ and
\[
\inf \left\{ \mult_{\xi}(D + (1/m)(\phi)) \mid m \in \ZZ_{>0}, \ \phi \in V_m  \setminus \{ 0 \}
\right\} = 0
\]
for $\xi \in X_K$,
then $\mu_{\QQ,\xi}(\overline{D}; V_{\bullet}) = 0$.
\end{Theorem}

\begin{proof}
First let us consider the following claim:

\begin{Claim}
\label{claim:thm:vol:cvol:mu:zero:01}
$\Theta(\overline{D}; V_{\bullet}) = \Delta(V_{\bullet})$.
\end{Claim}

\begin{proof}
It is sufficient to see that
$\Delta(V_{\bullet})^{\circ} \subseteq \left\{ x \in \Delta(V_{\bullet}) \mid G_{(\overline{D};V_{\bullet})}(x) \geq 0 \right\}$.
We assume the contrary, that is, there is $y \in  \Delta(V_{\bullet})^{\circ}$ with
$G_{(\overline{D};V_{\bullet})}(y) < 0$. Then, by using the upper semicontinuity of $G_{(\overline{D};V_{\bullet})}$,
we can find an open neighborhood $U$ of $y$ such that
$U \subseteq \Delta(V_{\bullet})^{\circ}$ and $G_{(\overline{D};V_{\bullet})}(x) < 0$ for all $x \in U$.
Then, as $\Theta(\overline{D}; V_{\bullet}) \subseteq  \Delta(V_{\bullet}) \setminus U$,
by the integral formulae of $\avol$ and $\acvol$ 
(cf. \eqref{eqn:integral:formula:vol}, \eqref{eqn:integral:formula:chivol}) and (1) in Lemma~\ref{lem:vol:asym:mu},
\begin{align*}
\frac{\acvol(\overline{D};V_{\bullet})}{(d+1)! [K : \QQ]} & = \int_{ \Delta(V_{\bullet})} G_{(\overline{D};V_{\bullet})}(x) dx =
 \int_{U} G_{(\overline{D};V_{\bullet})}(x) dx +  \int_{ \Delta(V_{\bullet}) \setminus U } G_{(\overline{D};V_{\bullet})}(x) dx \\
 & <
 \int_{ \Delta(V_{\bullet}) \setminus U } G_{(\overline{D};V_{\bullet})}(x) dx 
\leq  \int_{\Theta(\overline{D}; V_{\bullet})} G_{(\overline{D};V_{\bullet})}(x) dx = \frac{\avol(\overline{D};V_{\bullet})}{(d+1)! [K : \QQ]}.
\end{align*}
This is a contradiction.
\end{proof}

Let $B$ be the Zariski closure of $\{ \xi \}$ in $X$.
We choose $P \in X(\overline{K})$ and a local system of parameters $z_P = (z_1, \ldots, z_d)$ at $P$
such that $P$ is a regular point of $B_{\overline{K}}$ and 
$z_1 = \cdots = z_r = 0$ is a local equation of $B_{\overline{K}}$ at $P$.
Let $\nu : \RR^d \to \RR$ be the linear map given by
$\nu(x_1, \ldots, x_d) = x_1 + \cdots + x_r$.
We also choose a monomial order
$\precsim$ such that $\nu(a) \leq \nu(b)$ for all $a, b \in \ZZ_{\geq 0}^d$ with
$a \precsim b$.
By our assumption,
\[
\inf \left\{ \mult_{\xi}(D + (1/m)(\phi)) \mid m \in \ZZ_{>0}, \ \phi \in V_m  \setminus \{ 0 \}
\right\} = 0.
\]
This means that $\min \{ \nu(x) \mid x \in \Delta(V_{\bullet}) \} = 0$, 
and hence, by Claim~\ref{claim:thm:vol:cvol:mu:zero:01}
and (2) in Lemma~\ref{lem:vol:asym:mu},
\[
\mu_{\QQ, \xi}(\overline{D};V_{\bullet}) = \min \{ \nu(x) \mid x \in \Theta(\overline{D}; V_{\bullet}) \} = 0.
\]
\end{proof}

\begin{Corollary}
\label{cor:vol:cvol:mu:zero}
If $D_K$ is nef and big on the generic fiber $X_{K}$ and
$\avol(\overline{D}) = \acvol(\overline{D}) > 0$, then
$\mu_{\QQ, \xi}(\overline{D}) = 0$ for all $\xi \in X_{K}$.
\end{Corollary}

\begin{proof}
As $D_{K}$ is nef and big, in the similar way as \cite[Proposition~6.5.3]{MoArZariski},
for any $\epsilon > 0$, there is $\phi \in \Rat(X_{K})^{\times}_{\QQ}$ such that
\[
D_K + (\phi)_{\QQ} \geq 0\quad\text{and}\quad
\mult_{\xi}(D_K + (\phi)_{\QQ}) < \epsilon,
\]
which means that
\[
\inf \left\{ \mult_{\xi}(D + (1/m)(\phi)) \mid m \in \ZZ_{>0}, \ \phi \in H^0(X_K, mD_K)  \setminus \{ 0 \}
\right\} = 0.
\]
Thus the corollary follows from Theorem~\ref{thm:vol:cvol:mu:zero}.
\end{proof}

\section{Equality condition for the generalized Hodge index theorem}

Here let us give the proof of the main theorem of this paper.
We assume that $d=1$.
Let us begin with the following two lemmas.

\begin{Lemma}
\label{lem:nef:deg:zero}
We assume that  $X$ is regular.
For an integrable arithmetic $\RR$-Cartier divisor $\overline{D}$ of $C^0$-type on $X$
\rom{(}cf. Conventions and terminology~\rom{\ref{CT:Div:Nef}}\rom{)},
we have the following:
\begin{enumerate}
\renewcommand{\labelenumi}{(\arabic{enumi})}
\item
We assume that $\deg(D_K) = 0$.
Then $\adeg(\overline{D}^2) = 0$ if and only if
$\overline{D} = \widehat{(\psi)}_{\RR} + (0, \lambda)$
for some $\psi \in \Rat(X)^{\times}_{\RR}$ and $\lambda \in \RR$.

\item
The following are equivalent:
\begin{enumerate}
\renewcommand{\labelenumii}{(\alph{enumii})}
\item
$\deg(D_K) = 0$ and $\overline{D}$ is nef.

\item
$\deg(D_K) = 0$, $\overline{D}$ is pseudo-effective and $\adeg(\overline{D}^2) = 0$.
\end{enumerate}
\end{enumerate}
\end{Lemma}

\begin{proof}
(1) First we assume that $\adeg(\overline{D}^2) = 0$.
By \cite[Theorem~2.2.3, Remark~2.2.4]{MoD},
there are $\phi \in \Rat(X)^{\times}_{\RR}$ and an $F_{\infty}$-invariant locally constant 
real valued function
$\eta$ on $X(\CC)$ such that
$\overline{D} = \widehat{(\phi)}_{\RR} + (0, \eta)$.
Let $K(\CC)$ be the set of all embeddings $\sigma : K \hookrightarrow \CC$.
For each $\sigma \in K(\CC)$, we set $X_{\sigma} = X \times^{\sigma}_{\Spec(O_K)} \Spec(\CC)$, where 
$\times^{\sigma}_{\Spec(O_K)}$ means
the fiber product with respect to $\sigma : K \hookrightarrow \CC$.
Note that $\{ X_{\sigma} \}_{\sigma \in K(\CC)}$ gives rise to all connected components of $X(\CC)$.
Let $\eta_{\sigma}$ be the value of $\eta$ on $X_{\sigma}$.
We set
\[
\lambda = \frac{1}{[K : \QQ]} \sum_{\sigma \in K(\CC)} \eta_{\sigma}\quad\text{and}\quad
\xi = \eta - \lambda.
\]
Then $\xi_{\bar{\sigma}} = \xi_{\sigma}$ for all $\sigma \in K(\CC)$
and $\sum_{\sigma \in K(\CC)} \xi_{\sigma} = 0$.
Thus, by Dirichlet's unit theorem,
there is $u \in O_K^{\times} \otimes \RR$ such that $\widehat{(u)}_{\RR} = (0, \xi)$.
Therefore, we have 
\[
\overline{D} = \widehat{(\phi u)}_{\RR} + (0, \lambda).
\]
The converse is obvious.

\medskip
(2) (a) $\Longrightarrow$ (b) follows from the non-negativity of $\adeg(\overline{D}^2)$
(\cite[Proposition~6.4.2]{MoArZariski}, \cite[SubSection~2.1]{MoD}) and 
the Hodge index theorem (\cite[Theorem~2.2.3]{MoD}).
Let us show that (b) $\Longrightarrow$ (a).
By (1), 
$\overline{D} = \widehat{(\psi)}_{\RR} + (0, \lambda)$
for some $\psi \in \Rat(X)^{\times}_{\RR}$ and $\lambda \in \RR$.
Let $\overline{A}$ be an ample arithmetic Cartier divisor of $C^{\infty}$-type.
Then, as $\overline{D}$ is pseudo-effective,
\[
0 \leq \adeg(\overline{A} \cdot \overline{D}) = \frac{\lambda [K : \QQ] \deg(A_K)}{2},
\]
and hence $\lambda \geq 0$, as required.
\end{proof}

\begin{Lemma}
\label{lem:vol:formula:epsilon}
In this lemma, $X$ is not necessarily an arithmetic surface, that is,
$X$ is a $(d+1)$-dimensional, generically smooth, normal and projective arithmetic variety.
Let $\overline{D}$ be an arithmetic $\RR$-Cartier divisor of $C^0$-type on $X$.
Then,
\[
\avol(\overline{D}) \leq
\avol(\overline{D} + (0, \epsilon)) \leq \avol(\overline{D}) + \frac{\epsilon (d+1)[K : \QQ] \vol(D_K)}{2}
\]
for $\epsilon \in \RR_{\geq 0}$.
\end{Lemma}

\begin{proof}
The first inequality is obvious.
Note that $\Vert\cdot\Vert_{m(\overline{D} + (0,\epsilon))} = e^{-\frac{m\epsilon}{2}} \Vert\cdot\Vert_{m\overline{D}}$
for all $m \geq 0$. Thus, by using \cite[(3) in Proposition~2.1]{MoCont}, there is a constant $C$ such that
\begin{multline*}
\frac{\log \# \aH(X, m (\overline{D} + (0, \epsilon)))}{m^{d+1}/(d+1)!} \leq
\frac{\log \# \aH(X, m \overline{D})}{m^{d+1}/(d+1)!} \\
+ \frac{\epsilon(d+1)[K : \QQ]}{2} \frac{\dim_K H^0(X_K, mD_K)}{m^d/d!}
+ C \frac{\log m}{m}
\end{multline*}
holds for $m \gg 1$.
Thus the second inequality follows.
\end{proof}

The following theorem is the main result of this paper.

\begin{Theorem}
\label{thm:equal:cond:gen:Hodge}
Let $\overline{D}$ be an integrable arithmetic $\RR$-Cartier divisor of $C^0$-type on $X$
with $\deg(D_K) > 0$. Then
$\adeg (\overline{D}^2) = \avol(\overline{D})$ if and only if
$\overline{D}$ is nef.
\end{Theorem}

\begin{proof}
Let $\nu : X' \to X$ be a desingularization of $X$ (cf. \cite{Lip}).
Then $\adeg(\nu^*(\overline{D})^2) = \adeg(\overline{D}^2)$ and
$\avol(\nu^*(\overline{D})) = \avol(\overline{D})$.
Moreover, 
$\nu^*(\overline{D})$ is nef if and only if
$\overline{D}$ is nef. Therefore, we may assume that $X$ is regular.

By \cite[Corollary~5.5]{MoCont} and
\cite[Proposition-Definition~6.4.1]{MoArZariski},
if $\overline{D}$ is nef, then $\adeg (\overline{D}^2) = \avol(\overline{D})$,
so that we need to show that if $\adeg (\overline{D}^2) = \avol(\overline{D})$,
then $\overline{D}$ is nef.

First we assume that $\overline{D}$ is big.
Note that 
\[
\adeg(\overline{D}^2) \leq \acvol(\overline{D}) \leq  \avol(\overline{D}).
\]
Thus, by Theorem~\ref{thm:cvol:self:ineq} and Corollary~\ref{cor:vol:cvol:mu:zero},
$\overline{D}$ is relatively nef and
$\mu_{\RR, \xi}(\overline{D}) = 0$ for $\xi \in X_{K}$.
By \cite[Theorem~9.2.1]{MoArZariski},
there is a greatest element $\overline{P}$ 
of $\Upsilon(\overline{D})$ (cf. Conventions and terminology~\ref{CT:positivity:arithmetic:divisors}).
If we set $\overline{N} := \overline{D} - \overline{P}$, then
$\overline{D} = \overline{P} + \overline{N}$ is a Zariski decomposition of $\overline{D}$ 
(cf. Proposition~\ref{prop:small:sec:D:P}).
Then, 
by \cite[Claim~9.3.5.1]{MoArZariski} or \cite[Theorem~4.1.1]{MoArLinB},
\[
\mult_{\xi}(N) = \mu_{\RR, \xi}(\overline{D}) = 0
\]
for all $\xi \in X_K$,
which implies that $N$ is vertical.
In particular, $\adeg(\rest{\overline{D}}{C}) \geq 0$ for all horizontal reduced and irreducible $1$-dimensional closed subschemes $C$ on $X$, and
hence $\overline{D}$ is nef because $\overline{D}$ is relatively nef.

Next we assume that $\overline{D}$ is not big.
Then $\adeg(\overline{D}^2) = \avol(\overline{D}) = 0$.
Thus, for $\epsilon \in \RR_{>0}$,
\[
\epsilon[K : \QQ]
\deg(D_K) = 
\adeg((\overline{D} + (0, \epsilon))^2)
\leq \avol(\overline{D} + (0, \epsilon)) \leq \epsilon[K : \QQ]
\deg(D_K)
\]
by the generalized Hodge index theorem (cf. Theorem~\ref{thm:cvol:self:ineq}) and 
Lemma~\ref{lem:vol:formula:epsilon}, and hence $\overline{D} + (0, \epsilon)$ is big and
$\adeg((\overline{D} + (0, \epsilon))^2) =
\avol(\overline{D} + (0, \epsilon))$. Therefore, by the previous observation,
$\overline{D} + (0, \epsilon)$ is nef for all $\epsilon \in \RR_{>0}$,
which means that $\overline{D}$ is nef.
\end{proof}

As a corollary of the above theorem, we have the following:

\begin{Corollary}
\label{cor:characterization:nef:ar:div}
Let $\overline{D}$ be an integrable arithmetic $\RR$-Cartier divisor of $C^0$-type on $X$.
Then $\overline{D}$ is nef if and only if
$\overline{D}$ is pseudo-effective  
and
$\adeg (\overline{D}^2) = \avol(\overline{D})$.
\end{Corollary}

\begin{proof}
We need to show that if $\overline{D}$ is pseudo-effective and
$\adeg (\overline{D}^2) = \avol(\overline{D})$, then $\overline{D}$ is nef.
Clearly $\deg(D_K) \geq 0$.
If $\deg(D_K) > 0$, then the nefness of $\overline{D}$ follows from Theorem~\ref{thm:equal:cond:gen:Hodge}.
Moreover, if $\deg(D_K) = 0$, then (2) in Lemma~\ref{thm:equal:cond:gen:Hodge} implies
the assertion.
\end{proof}

\section{Negative part of Zariski decomposition}

We assume that $d=1$.
As an application of Theorem~\ref{thm:equal:cond:gen:Hodge},
let us see that 
the self-intersection number of the negative part of a Zariski decomposition
is negative.

\begin{Theorem}
\label{thm:strict:negative:negative:part}
Let $\overline{D}$ be
an integrable arithmetic $\RR$-Cartier divisor of $C^0$-type on $X$
such that 
$\deg(D_K) \geq 0$.
Let $\overline{D} = \overline{P} + \overline{N}$ be a Zariski decomposition of $\overline{D}$
\rom{(}cf. Conventions and terminology~\rom{\ref{CT:positivity:arithmetic:divisors}}\rom{)}.
Then $\adeg(\overline{N}^2) < 0$ if and only if $\overline{D}$ is not nef.
\end{Theorem}

\begin{proof}
First of all, note that $\overline{D}$ is pseudo-effective.
As $\adeg (\overline{P} \cdot \overline{N}) = 0$ by the following Lemma~\ref{lem:perpendicular:positive:negative},
\[
\avol(\overline{D}) - \adeg(\overline{D}^2) = \avol(\overline{P}) - \adeg(\overline{D}^2)
= \adeg(\overline{P}^2) - \adeg(\overline{D}^2) = -\adeg(\overline{N}^2).
\]
Therefore, by Corollary~\ref{cor:characterization:nef:ar:div},
$\overline{D}$ is not nef if and only if $\avol(\overline{D}) > \adeg(\overline{D}^2)$.
Thus the assertion follows.
\end{proof}

\begin{Lemma}
\label{lem:perpendicular:positive:negative}
Let $\overline{D}$ be an integrable arithmetic $\RR$-Cartier divisor of $C^0$-type on $X$.
If $\overline{D} = \overline{P} + \overline{N}$ is a Zariski decomposition of $\overline{D}$,
then $\adeg(\overline{P} \cdot \overline{N}) = 0$ and $\adeg(\overline{N}^2) \leq 0$.
\end{Lemma}

\begin{proof}
For $0 < \epsilon \leq 1$, we set $\overline{D}_{\epsilon} =\overline{P} + \epsilon \overline{N}$.
Then $\overline{D}_{\epsilon}$ is integrable and
$\avol(\overline{P}) = \avol(\overline{D}_{\epsilon})$ because
\[
\avol(\overline{P}) \leq \avol(\overline{D}_{\epsilon}) \leq \avol(\overline{D}) = \avol(\overline{P}).
\]
Thus, by the generalized Hodge index theorem (cf. Theorem~\ref{thm:cvol:self:ineq}),
\[
\adeg ((\overline{P} + \epsilon \overline{N})^2) = \adeg (\overline{D}_{\epsilon}^2) \leq \avol(\overline{D}_{\epsilon}) = \avol(\overline{P}) =
\adeg(\overline{P}^2),
\]
and hence
\[
2 \adeg(\overline{P} \cdot \overline{N}) + \epsilon \adeg(\overline{N}^2) \leq 0.
\]
In particular, $\adeg(\overline{P} \cdot \overline{N}) \leq 0$.
On the other hand, as $\overline{P}$ is nef and $\overline{N}$ is effective,
$\adeg (\overline{P} \cdot \overline{N}) \geq 0$.
Thus $\adeg(\overline{P} \cdot \overline{N}) = 0$ and $\adeg(\overline{N}^2) \leq 0$.
\end{proof}

\begin{Remark}
If $\overline{D}$ is big,
then the Zariski decomposition $\overline{D} = \overline{P} + \overline{N}$ is
uniquely determined by \cite[Theorem~4.2.1]{MoArLinB}.
Otherwise, it is not necessarily unique.
\end{Remark}

As a consequence of the above theorem,
we have the following numerical characterization of the greatest element of $\Upsilon(\overline{D})$
(cf. Conventions and terminology~\ref{CT:positivity:arithmetic:divisors}).

\begin{Corollary}
\label{cor:characterization:Zariski:decomp}
We assume that $X$ is regular. Let $\overline{D}$ and
$\overline{P}$ be arithmetic $\RR$-Cartier divisors of $C^0$-type on $X$.
Then the following are equivalent:
\begin{enumerate}
\renewcommand{\labelenumi}{(\arabic{enumi})}
\item
$\overline{P}$ is the greatest element of $\Upsilon(\overline{D})$, that is,
$\overline{P} \in \Upsilon(\overline{D})$ and $\overline{M} \leq \overline{P}$ 
for all $\overline{M} \in \Upsilon(\overline{D})$.

\item
$\overline{P}$ is an element of $\Upsilon(\overline{D})$ with the following property:
\[
\adeg(\overline{P} \cdot \overline{B}) = 0\quad\text{and}\quad
\adeg(\overline{B}^2) < 0
\]
for all
integrable arithmetic $\RR$-Cartier divisors $\overline{B}$ of $C^0$-type
with
$(0,0) \lneqq \overline{B} \leq \overline{D} - \overline{P}$.
\end{enumerate}
\end{Corollary}

\begin{proof}
(1) $\Longrightarrow$ (2) :
By Proposition~\ref{prop:small:sec:D:P},
$\avol(\overline{D}) = \avol(\overline{P})$, so that $\overline{P} + \overline{B}$ is a Zariski decomposition
because
\[
\avol(\overline{P}) \leq \avol(\overline{P} + \overline{B}) \leq \avol(\overline{D}).
\]
Thus
$\adeg(\overline{P} \cdot \overline{B}) = 0$ 
by Lemma~\ref{lem:perpendicular:positive:negative}. 
As $\overline{B} \gneqq (0,0)$ and $\overline{P}$ is the greatest element of $\Upsilon(\overline{D})$,
$\overline{P} + \overline{B}$ is not nef, so that
$\adeg(\overline{B}^2) < 0$ by Theorem~\ref{thm:strict:negative:negative:part}.

(2) $\Longrightarrow$ (1) :
Let $\overline{M}$ be an element of $\Upsilon(\overline{D})$.
If we set $\overline{A} = \max \{ \overline{P}, \overline{M} \}$ 
(cf. Conventions and terminology~\ref{CT:max:arith:div}) 
and $\overline{B} = \overline{A} - \overline{P}$,
then $\overline{B}$ is effective, $\overline{A} \leq \overline{D}$ and $\overline{A}$ is nef by \cite[Lemma~9.1.2]{MoArZariski}. 
Moreover, 
\[
\overline{B} = \overline{A} - \overline{P} \leq \overline{D} - \overline{P}.
\]
If we assume $\overline{B} \gneqq (0,0)$, then,
by the property, $\adeg(\overline{P} \cdot \overline{B}) = 0$ and $\adeg(\overline{B}^2) < 0$.
On the other hand, as $\overline{A}$ is nef and $\overline{B}$ is effective,
\[
0 \leq \adeg(\overline{A} \cdot \overline{B}) = \adeg(\overline{P} + \overline{B} \cdot \overline{B })
= \adeg(\overline{B}^2),
\]
which is a contradiction, so that $\overline{B} = (0,0)$, that is,
$\overline{P} = \overline{A}$, which means that $\overline{M} \leq \overline{P}$, as required.
\end{proof}

\begin{Corollary}
We assume that $X$ is regular.
Let $\overline{D}$ be an arithmetic $\RR$-Cartier divisor of $C^0$-type on $X$ such that
$\Upsilon(\overline{D}) \not= \emptyset$.
Let $\overline{P}$ be the greatest element of $\Upsilon(\overline{D})$
\rom{(}cf. \cite[Theorem~9.2.1]{MoArZariski}\rom{)} and let $\overline{N} := \overline{D} - \overline{P}$.
We assume that $N \not= 0$. Let $N = c_1 C_1 + \cdots + c_l C_l$ be the decomposition
such that $c_1, \ldots, c_l \in \RR_{>0}$ and
$C_1, \ldots, C_l$ are distinct reduced and irreducible $1$-dimensional closed subschemes on $X$.
Let $\overline{C}_1 = (C_1, h_1), \ldots, \overline{C}_l = (C_l, h_l)$ be effective arithmetic
Cartier divisors of $C^0$-type such that
such that $c_1(\overline{C}_1), \ldots, c_1(\overline{C}_l)$ are positive currents and
\[
c_1 \overline{C}_1 + \cdots + c_l \overline{C}_l \leq \overline{N}.
\]
Then
\[
\adeg(\overline{P} \cdot \overline{C_1}) = \cdots = \adeg(\overline{P} \cdot \overline{C_l}) = 0
\]
and
the $(l\times l)$ symmetric matrix given by
\[
\left( \adeg(\overline{C}_i \cdot \overline{C}_j) \right)_{\substack{1 \leq i \leq l \\ 1 \leq j \leq l}}
\]
is negative definite.
\end{Corollary}

\begin{proof}
For $x = (x_1, \ldots, x_l) \in \RR^l$, we set
$\overline{B}_x = x_1 \overline{C}_1 + \cdots + x_l \overline{C}_l$
and $\overline{D}_x = \overline{P} + \overline{B}_x$.
If $0 \leq x_i \leq c_i$ for all $i = 1, \ldots, l$,
then $\overline{B}_x$ is integrable and $(0,0) \leq \overline{B}_x \leq \overline{N}$.
Thus, by Corollary~\ref{cor:characterization:Zariski:decomp},
\[
0 = \adeg(\overline{P} \cdot \overline{B}_{(c_1, \ldots, c_l)}) =
c_1 \adeg(\overline{P} \cdot \overline{C}_1) + \cdots + c_l \adeg(\overline{P} \cdot \overline{C}_l).
\]
Note that $\adeg(\overline{P} \cdot \overline{C}_i) \geq 0$ for all $i=1, \ldots, l$.
Therefore,
\[
\adeg(\overline{P} \cdot \overline{C_1}) = \cdots = \adeg(\overline{P} \cdot \overline{C_l}) = 0
\]
Here we claim the following:

\begin{Claim}
If $x \in (\RR_{\geq 0})^l \setminus \{0 \}$, then $\adeg( \overline{B}_x^2) < 0$.
\end{Claim}

\begin{proof}
Note that $\overline{B}_{tx} = t \overline{B}_x$ and that
we can find a positive number $t$ with $t x_i \leq c_i$ ($\forall i$).
Thus we may assume that $x_i \leq c_i$ ($\forall i$), and hence
the assertion follows by Corollary~\ref{cor:characterization:Zariski:decomp}.
\end{proof}

We need to see that if $x \in \RR^l \setminus \{ 0 \}$,
then $\adeg (\overline{B}_x^2) < 0$.
We can choose 
\[
y = (y_1, \ldots, y_l), z = (z_1, \ldots, z_l) \in (\RR_{\geq 0})^l
\]
such that
$x = y - z$ and $\{ i \mid y_i \not= 0 \} \cap \{ j \mid z_j \not= 0 \} = \emptyset$.
Note that either $y \not= 0$ or $z \not= 0$.
Moreover, $\adeg (\overline{B}_y \cdot \overline{B}_z) \geq 0$ because
$\overline{B}_y \geq (0,0)$, $\overline{B}_z \geq (0,0)$,
$c_1(\overline{B}_y)$ and $c_1(\overline{B}_z)$ are positive currents, 
and
$B_y$ and $B_z$ have no common reduced and irreducible $1$-dimensional closed subschemes. 
Thus, by using the above claim,
\[
\adeg (\overline{B}_x^2) = \adeg ((\overline{B}_y - \overline{B}_z)^2) 
= \adeg (\overline{B}_y^2) + \adeg(\overline{B}_z^2) - 2 \adeg(\overline{B}_y \cdot \overline{B}_z)
< 0.
\]
\end{proof}

\begin{Remark}
By \cite[Theorem~9.3.4, (4.1)]{MoArZariski},
we can find effective arithmetic Cartier divisors
$\overline{C}_1, \ldots, \overline{C}_l$ of $C^0$-type such that
$c_1(\overline{C}_1), \ldots, c_1(\overline{C}_l)$ are positive currents and
$c_1 \overline{C}_1 + \cdots + c_l \overline{C}_l \leq \overline{N}$.
\end{Remark}

\begin{Example}
Let $\PP^1_{\ZZ} = \Proj(\ZZ[T_0, T_1])$ and $H_i = \{ T_i = 0 \}$ for $i=0,1$.
We fix  positive numbers $a_0, a_1$ such that $a_0 < 1$, $a_1 < 1$ and $a_0 + a_1 \geq 1$.
Let us consider an arithmetic Cartier divisor $\overline{D}$ of $C^{\infty}$-type given by
\[
\overline{D} := (H_0, \log(a_0 + a_1 \vert z \vert^2)),
\]
where $z = T_1/T_0$. 
Note that $c_1(\overline{D})$ is a positive form.
Moreover, $\overline{D}$ is pseudo-effective and not nef (cf. \cite[Theorem~2.3]{MoBig}).
In \cite[Theorem~4.1]{MoBig}, we give the greatest element of $\Upsilon(\overline{D})$ as follows:
Let $\varphi$ be a continuous function on the interval $[0,1]$ given by
\[
\varphi(x) = - (1-x)\log(1-x) -x \log(x)  + (1-x) \log (a_0) + x \log(a_1),
\]
and let $\vartheta = \min \{ x \in [0,1] \mid \varphi(x) \geq 0 \}$ and
$\theta = \max \{ x \in [0,1] \mid \varphi(x) \geq 0 \}$.
We set 
\[
\overline{P} := (\theta H_0 - \vartheta H_1, p(z)),\quad
\overline{N}_1 := (\vartheta H_1, n_1(z))\quad\text{and}\quad
\overline{N}_2 := ((1-\theta)H_0, n_2(z)),
\] 
where $p(z)$, $n_1(z)$ and $n_2(z)$ are Green functions given by
{\allowdisplaybreaks%
\begin{align*}
p(z) & :=\begin{cases}
\vartheta \log \vert z \vert^2 & \hskip7.4em \text{if $\vert z \vert \leq \sqrt{\frac{a_0\vartheta}{a_1(1-\vartheta)}}$}, \\
\log (a_0   + a_1\vert z \vert^2) & \hskip7.4em
\text{if $\sqrt{\frac{a_0\vartheta}{a_1(1-\vartheta)}} \leq \vert z \vert \leq \sqrt{\frac{a_0\theta}{a_1(1-\theta)}}$}, \\
\theta \log \vert z \vert^2 & \hskip7.4em \text{if $\vert z \vert \geq \sqrt{\frac{a_0\theta}{a_1(1-\theta)}}$}.
\end{cases} \\
n_1(z) & :=\begin{cases}
\log(a_0 + a_1 \vert z \vert^2) - \vartheta \log \vert z \vert^2 
& \hskip2.7em \text{if $\vert z \vert \leq \sqrt{\frac{a_0\vartheta}{a_1(1-\vartheta)}}$}, \\
0  & \hskip2.7em
\text{if $\vert z \vert \geq \sqrt{\frac{a_0\vartheta}{a_1(1-\vartheta)}}$}. \\
\end{cases} \\
n_2(z) & :=\begin{cases}
0 & 
\text{if $\vert z \vert \leq \sqrt{\frac{a_0\theta}{a_1(1-\theta)}}$}, \\
\log(a_1 + a_0 \vert z \vert^{-2}) + (1 - \theta) \log \vert z \vert^2 & \text{if $\vert z \vert \geq \sqrt{\frac{a_0\theta}{a_1(1-\theta)}}$}.
\end{cases}
\end{align*}}
Then $\overline{P}$ gives the greatest element of $\Upsilon(\overline{D})$ and 
$\overline{D} = \overline{P} + (\overline{N}_1 + \overline{N}_2)$.
It is easy to see that
\[
\adeg (\overline{P} \cdot \overline{N}_1) = \adeg (\overline{P} \cdot \overline{N}_2) = 0\quad\text{and}\quad
\adeg(\overline{N}_1 \cdot \overline{N}_2) = 0.
\]
Moreover,
\begin{align*}
\adeg(\overline{N}_1 \cdot \overline{N}_1) & =
\adeg(\overline{N}_1 \cdot (\overline{N}_1 -\vartheta \widehat{(z)} ))
= \adeg(\overline{N}_1 \cdot (\vartheta H_0, n_1(z) + \vartheta \log \vert z \vert^2 )) \\
& = \vartheta \adeg(\rest{\overline{N}_1}{H_0}) + 
\frac{1}{2}\int_{\PP^1(\CC)} c_1(\overline{N}_1)(n_1(z) + \vartheta \log \vert z \vert^2) \\
& = \frac{1}{2} \int_{\vert z \vert \leq \sqrt{\frac{a_0\vartheta}{a_1(1-\vartheta)}}} 
dd^c(\log (a_0 + a_1 \vert z \vert^2)) \log (a_0 + a_1 \vert z \vert^2) \\
& = \frac{(1 - \vartheta) \log(1 - \vartheta) + (\log (a_0) + 1) \vartheta}{2}.
\end{align*}
In the same way,
\[
\adeg(\overline{N}_2 \cdot \overline{N}_2) = \frac{\theta \log(\theta) + (\log (a_1) + 1) (1 - \theta)}{2}.
\]
Thus the negative definite symmetric matrix $(\adeg(\overline{N}_i \cdot \overline{N}_j))_{i,j=1,2}$
is
\[
\begin{pmatrix}
\frac{(1 - \vartheta) \log(1 - \vartheta) + (\log (a_0) + 1) \vartheta}{2} & 0 \\
0 & \frac{\theta \log(\theta) + (\log (a_1) + 1) (1 - \theta)}{2}
\end{pmatrix}.
\]
\end{Example}

\renewcommand{\thesection}{Appendix~\Alph{section}}
\renewcommand{\theTheorem}{\Alph{section}.\arabic{Theorem}}
\renewcommand{\theClaim}{\Alph{section}.\arabic{Theorem}.\arabic{Claim}}
\renewcommand{\theequation}{\Alph{section}.\arabic{Theorem}.\arabic{Claim}}

\setcounter{section}{0}

\section{Relative Zariski decomposition and pseudo-effectivity}
We assume that $X$ is regular and $d=1$.
Let $\overline{D} = (D, g)$ be an arithmetic $\RR$-Cartier divisor of $C^0$-type on $X$.
In this appendix,
we would like to investigate the pseudo-effectivity of the relative Zariski decomposition.

\begin{Proposition}
\label{prop:relative:Zariski:pseudo:effective}
We assume that $\deg(D_K) \geq 0$.
Let $\overline{Q}$ be the greatest element of $\Upsilon_{rel}(\overline{D})$
\rom{(}cf. Section~\rom{\ref{sec:relative:Zariski:decomp}}\rom{)}.
Then $\overline{D}$ is pseudo-effective if and only if
$\overline{Q}$ is pseudo-effective.
\end{Proposition}

\begin{proof}
It is obvious that if $\overline{Q}$ is pseudo-effective, then
$\overline{D}$ is also pseudo-effective, so that we assume that $\overline{D}$ is pseudo-effective.

First we consider the case where $\deg(D_K) > 0$.
Then, by \cite[Proposition~6.3.3]{MoArZariski},
$\overline{D} + (0, \epsilon)$ is big for any $\epsilon \in \RR_{>0}$.
By the property (d) in Theorem~\ref{thm:relative:Zariski:decomp}, 
the natural inclusion map $H^0(X, nQ) \to H^0(X, nD)$ is bijective and
$\Vert \cdot \Vert_{n\overline{Q}} = \Vert \cdot \Vert_{n\overline{D}}$ for each $n \geq 0$. 
Moreover, as
\[
\Vert \cdot \Vert_{n(\overline{Q} + (0, \epsilon))} = e^{-n\epsilon/2} \Vert \cdot \Vert_{n\overline{Q}}
\quad\text{and}\quad
\Vert \cdot \Vert_{n(\overline{D} + (0, \epsilon))} = e^{-n\epsilon/2} \Vert \cdot \Vert_{n\overline{D}},
\]
we have $\Vert \cdot \Vert_{n(\overline{Q} + (0, \epsilon))} = \Vert \cdot \Vert_{n(\overline{D} + (0, \epsilon))}$, and
hence $\overline{Q} + (0, \epsilon)$ is big for all $\epsilon \in \RR_{>0}$.
Thus the assertion follows.

Next we assume that $\deg(D_K) = 0$.
By \cite[Theorem~2.3.3]{MoD},
there are $\phi \in \Rat(X)^{\times}_{\RR}$, a vertical effective $\RR$-Cartier divisor $E$ on $X$ and
an $F_{\infty}$-invariant continuous function $\eta$ on $X(\CC)$ such that
$\overline{D} = \widehat{(\phi)}_{\RR} + (E, \eta)$ and 
$\pi^{-1}(P)_{red} \not\subseteq \Supp(E)$ for all $P \in \Spec(O_K)$.
For each embedding $\sigma : K \hookrightarrow \CC$, let $X_{\sigma} = X \times^{\sigma}_{\Spec(O_K)} \Spec(\CC)$ and
let $\lambda_{\sigma} = \min_{x \in X_{\sigma}} \{ \eta(x) \}$.
Note that $\lambda_{\bar{\sigma}} = \lambda_{\sigma}$ for all $\sigma$.
Let $\lambda : X(\CC) \to \RR$ be the local constant function such that the value of $\lambda$ on $X_{\sigma}$ is
$\lambda_{\sigma}$.

Here let us see that $\overline{Q} = \widehat{(\phi)}_{\RR} + (0, \lambda)$ is the greatest element of
$\Upsilon_{rel}(\overline{D})$. Otherwise,
there is an integrable arithmetic $\RR$-Cartier divisor $\overline{B} = (B, b)$ of $C^0$-type
such that
$(0,0) \lneqq \overline{B} \leq \overline{D} - \overline{Q} = (E, \eta - \lambda)$
and $\overline{Q} + \overline{B}$ is relatively nef.
Since $b$ is continuous and
\[
dd^c([b]) = c_1(\overline{B}) = c_1(\overline{Q} + \overline{B})
\]
is a positive current, $b$ is plurisubharmonic on $X(\CC)$, that is, $b$ is a locally constant function.
Let $b_{\sigma}$ be the value of $b$ on $X_{\sigma}$. If we choose $x_{\sigma} \in X_{\sigma}$ with
$\lambda_{\sigma} = \eta(x_{\sigma})$,
then
\[
0 \leq b_{\sigma} \leq \eta(x_{\sigma}) - \lambda_{\sigma} = 0,
\]
and hence $b = 0$, so that, as $\overline{Q} + \overline{B}$ is relatively nef,
\[
0 \leq \adeg(\overline{Q} + \overline{B} \cdot \overline{B}) = \adeg((B, 0)^2).
\]
On the other hand, by Zariski's lemma, $\adeg((B, 0)^2) < 0$.
This is a contradiction.

By \cite[Lemma~2.3.4 and Lemma~2.3.5]{MoD},
$(E, \lambda)$ is pseudo-effective. 
On the other hand,
by the following Lemma~\ref{lem:exist:nef:zero},
there is a nef arithmetic $\RR$-Cartier divisor $\overline{L}$ of $C^{\infty}$-type such that
$\deg(L_K) > 0$ and
$\adeg(\overline{L} \cdot (E, 0)) = 0$.
Thus, 
\[
0 \leq \adeg(\overline{L} \cdot (E, \lambda)) = \sum_{\sigma} \frac{\deg(L_K)\lambda_{\sigma}}{2},
\]
and hence $\sum_{\sigma} \lambda_{\sigma} \geq 0$.
We set $\lambda' = (1/[K : \QQ]) \sum_{\sigma} \lambda_{\sigma}$ and $\xi = \lambda - \lambda'$.
Then $\lambda' \geq 0$, $\sum_{\sigma} \xi_{\sigma} = 0$ and $\xi_{\bar{\sigma}} = \xi_{\sigma}$ for all $\sigma$, 
where $\xi_{\sigma}$ is the value of $\xi$ on $X_{\sigma}$. Thus,
by Dirichlet's unit theorem,
$(0, \xi) = \widehat{(u)}_{\RR}$ for some $u \in O_K^{\times} \otimes \RR$.
Therefore,
\[
\overline{Q} = \widehat{(\phi u)}_{\RR} + (0, \lambda'),
\]
which is pseudo-effective.
\end{proof}

\begin{Lemma}
\label{lem:exist:nef:zero}
Let $C_1, \ldots, C_r$ be vertical reduced and irreducible $1$-dimensional closed subschemes on $X$ such that
$\pi^{-1}(P)_{red} \not\subseteq C_1 \cup \cdots \cup C_r$ for all
$P \in \Spec(O_K)$.
Then there is a nef arithmetic $\RR$-Cartier divisor $\overline{L}$ of $C^{\infty}$-type such that
$\deg(L_K) > 0$ and
$\adeg(\overline{L} \cdot (C_i, 0)) = 0$ for all $i=1, \ldots, r$.
\end{Lemma}

\begin{proof}
Let $\overline{A}$ be an ample arithmetic Cartier divisor of $C^{\infty}$-type.
By using Zariski's lemma, we can find a vertical effective $\RR$-Cartier divisor $E$ such that
\[
\adeg((E, 0) \cdot (C_i, 0)) = -\deg (\overline{A} \cdot (C_i, 0))
\]
for all $i = 1, \ldots, r$ and that $\adeg((E, 0) \cdot (C, 0)) \geq 0$ for all vertical
reduced and irreducible $1$-dimensional closed subschemes $C$ with
$C \not\in \{ C_1, \ldots, C_r \}$.
Thus, if we set $\overline{L} := \overline{A} + (E, 0)$,
then $\overline{L}$ is a nef arithmetic $\RR$-Cartier divisor of $C^{\infty}$-type, $\deg(L_K) > 0$ and
$\adeg(\overline{L} \cdot (C_i, 0)) = 0$ for all $i=1, \ldots, r$.
\end{proof}

\section{Small sections of arithmetic $\RR$-divisors}

Let $\overline{D}$ be an arithmetic $\RR$-Cartier divisor of $C^0$-type on $X$.
In this appendix, let us consider a generalization of \cite[Proposition~9.3.3]{MoArZariski}.
Its proof is much simpler than one of \cite[Proposition~9.3.3]{MoArZariski}.

\begin{Proposition}
\label{prop:small:sec:D:P}
Let $\overline{P}$ be the greatest element of $\Upsilon(\overline{D})$
\rom{(}cf. Conventions and terminology~\rom{\ref{CT:positivity:arithmetic:divisors}}\rom{)}.
Then, for $\phi \in \Rat(X)^{\times}_{\RR}$,
$\overline{D} + \widehat{(\phi)}_{\RR}$ is effective if and only if
$\overline{P} + \widehat{(\phi)}_{\RR}$ is effective.
In particular,
the natural inclusion maps
\[
\aH(X, n\overline{P}) \hookrightarrow \aH(X, n\overline{D}),\quad
\aH_{\QQ}(X, \overline{P}) \hookrightarrow \aH_{\QQ}(X, \overline{D})\quad\text{and}\quad
\aH_{\RR}(X, \overline{P}) \hookrightarrow \aH_{\RR}(X, \overline{D})
\]
are bijective for each $n \geq 0$.
\end{Proposition}

\begin{proof}
We assume that $\overline{D} + \widehat{(\phi)}_{\RR}$ is effective.
Then $-\widehat{(\phi)}_{\RR} \in \Upsilon(\overline{D})$, and hence
$-\widehat{(\phi)}_{\RR} \leq \overline{P}$, that is,
$\overline{P} + \widehat{(\phi)}_{\RR}$ is effective.
The converse is obvious.
\end{proof}

As a corollary of the above proposition, we have the following.

\begin{Corollary}
\label{cor:thm:small:sec:D:P}
We assume that $d = 1$.
Let $\overline{D} = \overline{P} + \overline{N}$ be a Zariski decomposition of $\overline{D}$
\rom{(}Conventions and terminology~\rom{\ref{CT:positivity:arithmetic:divisors}}\rom{)}.
If $\overline{D}$ is big, then
the natural inclusion maps
\[
\aH(X, n\overline{P}) \hookrightarrow \aH(X, n\overline{D}),\quad
\aH_{\QQ}(X, \overline{P}) \hookrightarrow \aH_{\QQ}(X, \overline{D})\quad\text{and}\quad
\aH_{\RR}(X, \overline{P}) \hookrightarrow \aH_{\RR}(X, \overline{D})
\]
are bijective for each $n \geq 0$.
\end{Corollary}

\begin{proof}
Let $\mu : X' \to X$ be a desingularization of $X$ (cf. \cite{Lip}).
Then 
\[
\mu^*(\overline{D}) = \mu^*(\overline{P}) + \mu^*(\overline{N})
\]
is a Zariski decomposition of $\mu^*(\overline{D})$.
Thus, by \cite[Theorem~4.2.1]{MoArLinB}, $\mu^*(\overline{P})$ gives the greatest element of
$\Upsilon(\mu^*(\overline{D}))$. Therefore, by Proposition~\ref{prop:small:sec:D:P},
\[
\aH(X', n\mu^*(\overline{P})) = \aH(X', n\mu^*(\overline{D}))\quad\text{and}\quad
\aH_{\KK}(X', \mu^*(\overline{P})) = \aH_{\KK}(X', \mu^*(\overline{D}))
\]
for each $n \geq 0$, where $\KK$ is either $\QQ$ or $\RR$.
Let us consider the following commutative diagrams:
\[
\begin{CD}
\aH(X, n\overline{P}) @>>> \aH(X', n\mu^*(\overline{P})) \\
@VVV @| \\
\aH(X, n\overline{D}) @>>> \aH(X', n\mu^*(\overline{D}))
\end{CD}
\qquad
\begin{CD}
\aH_{\KK}(X, \overline{P}) @>>> \aH_{\KK}(X', \mu^*(\overline{P})) \\
@VVV @| \\
\aH_{\KK}(X, \overline{D}) @>>> \aH_{\KK}(X', \mu^*(\overline{D}))
\end{CD}
\]
Note that each horizontal arrow is bijective.
Thus the assertions follows.
\end{proof}

\end{document}